%% file: main_arXiv.tex
\documentclass[12pt]{article}
\input{RT.sty}
\usepackage[margin=2cm]{geometry}

\usepackage{apacite,authblk}
\renewcommand{\cite}{\citeA}

\begin{document}
	\title{Rank-Constrained Least-Squares: Prediction and Inference}
	
    \author[1]{Michael Law}%\thanks{Supported in part by NSF Grant DMS-1646108}}
    \author[1]{Ya\hspace{-.1em}'\hspace{-.1em}acov Ritov}%\thanks{Supported in part by NSF Grants DMS-1712962 and DMS-2113364}}
    \author[2]{Ruixiang Zhang}%\thanks{Supported in part by NSF Grants DMS-1856541 and DMS-1926686 and the Ky Fan and Yu-Fen Fan Endowment Fund at the Institute for Advanced Study}}
    \author[1]{Ziwei Zhu}%\thanks{Supported by NSF grant DMS 2015366}}

    \affil[1]{University of Michigan, Ann Arbor}
    \affil[2]{University of California, Berkeley}
	
	\renewcommand\footnotemark{}
	
	\maketitle
	
	% Abstract
	\begin{abstract}
	    \input{abstract}
	\end{abstract}
	
	% Main
	\input{main}
	\input{appendix}
	
	% Acknowledgements and Funding
	\section{Acknowledgements}
	\input{acknowledgements}
	\input{funding}
	
	% References
	\bibliographystyle{newapa}
	\bibliography{RTref}
	
	% Supplement
	\newcounter{suppsection}
	\setcounter{suppsection}{1}
	\def\thesection{S\arabic{suppsection}}
	\section{Supplementanry Material}
	\input{supplement}

\end{document}

%% file: abstract.tex
In this work, we focus on the high-dimensional trace regression model with a low-rank coefficient matrix. We establish a nearly optimal in-sample prediction risk bound for the rank-constrained least-squares estimator under \emph{no assumptions} on the design matrix. Lying at the heart of the proof is a covering number bound for the family of projection operators corresponding to the subspaces spanned by the design.  By leveraging this complexity result, we perform a power analysis for a permutation test on the existence of a low-rank signal under the high-dimensional trace regression model. We show that the permutation test based on the rank-constrained least-squares estimator achieves non-trivial power with no assumptions on the minimum (restricted) eigenvalue of the covariance matrix of the design. Finally, we use alternating minimization to approximately solve the rank-constrained least-squares problem to evaluate its empirical in-sample prediction risk and power of the resulting permutation test in our numerical study. 

%% file: main.tex
	\section{Introduction}\label{sectionintroduction}
	
	In this work, we focus on the trace regression model: 
	\begin{align}\label{modeltrm}
		y = \langle \Xmat, \Thetastar \ranglehs + \epsilon. 
	\end{align}
	Here $y$ is a real-valued response, $\bX$ is a feature matrix valued in $\RR ^ {d_1 \times d_2}$, $\bTheta ^ * \in \RR ^ {d_1 \times d_2}$ is the parameter of our interest, and $\epsilon$ is a noise term that is independent of $\bX$. Throughout, objects with a superscript * denote true model paramters and we define $\langle\cdot, \cdot\rangle_{\mathrm{HS}}$ as the trace inner product in the sense that given any $\bA, \bB \in \RR ^ {d_1 \times d_2}$, $\langle\bA, \bB\rangle_{\mathrm{HS}} \defined \tr(\bA ^ \top \bB)$. 
	We write $d \defined \dtwo$ and assume without the loss of generality that $\done \leq \dtwo$ by possibly transposing the data.
% 	By possibly transposing the data, we assume without loss of generality that $\dtwo \leq \done$. \zzw{$d:= \max(d_1, d_2)$} 
	
	Suppose we have $n$ independent observations $(\bX_i, y_i)_{i \in [n]}$ generated from model \eqref{modeltrm}. Under a high-dimensional setup, $n$ is much smaller than $d_1d_2$, and some structural assumptions on $\bTheta ^ *$ are necessary to reduce the degrees of freedom of $\bTheta ^ *$ to achieve estimation consistency. Here, we assume that $\Thetastar$ is low-rank; that is, $\rstar \defined \rank(\Thetastar)$ with $\rstar d \ll n$. Given that one needs at most $(2r ^ * + 1)d$ parameters to determine $\bTheta ^ *$ through a singular value decomposition, intuitively a sample of size $n$ should suffice to achieve estimation consistency. The high-dimensional low-rank trace regression model was first introduced by \cite{rohde2011} and admits many special cases of wide interest. For instance, when $\Thetastar$ and $\bX$ are diagonal, model \eqref{modeltrm} reduces to a sparse linear regression model: 
	\begin{align}\label{modellm}
		y = \langle \xvec, \betastar \rangleeuclid + \epsilon, 
	\end{align}	
	where $\betastar = \mathrm{diag}(\Thetastar)$. Note that $\betastar$ is sparse because $\|\betastar\|_0 = r ^ * \ll d$. When $\bX$ is a singleton in the sense that $\bX = \be_i\be_j ^ \top$, where $\be_i$ and $\be_j$ are the $i$th and $j$th canonical basis vectors respectively, model \eqref{modeltrm} reduces to a low-rank matrix completion problem (\cite{candes2009exact}, \cite{koltchinskii2011}, \cite{recht2011simpler}, \cite{NWa12}).
	
	Perhaps the most natural approach to incorporate the low-rank structure in estimating $\Thetastar$ is to enforce a rank-constraint directly. Consider the following rank-constrained least-squares estimator: 
	\begin{align}
	    \label{eq:rank_ols}
		\Thetahatlzero(r) = \argmin_{\Thetamat \in \R^{\done \times \dtwo}, \rank(\Thetamat) \leq r} \sum_{i=1}^{n} \left( y_{i} - \langle \Xmat_{i}, \Thetamat \ranglehs \right)^{2}.
	\end{align}
	Note that the rank constraint is non-convex, thereby imposing a fundamental challenge computationally in obtaining this estimator. To resolve this issue, one can resort to nuclear-norm regularization to encourage low-rank structure of the estimator. Specifically, for some $\lambda > 0$, consider 
	\begin{align}
	    \label{eq:nuclear_ols}
		\Thetahatnn(\lambda) = \argmin_{\Thetamat \in \R^{\done \times \dtwo}} \Big\{ \frac{1}{n} \sum_{i=1}^{n} \left( y_{i} - \langle \Xmat_{i}, \Thetamat \ranglehs \right)^{2} + \lambda \Vert \Thetamat \Vertnuclear \Big\},
	\end{align}
	where $\Vert \cdot \Vertnuclear$ denotes the nuclear norm. Problem \eqref{eq:nuclear_ols} is convex and thus amenable to polynomial-time algorithms. The past decade or so has witnessed a flurry of works on statistical guarantees for $\Thetahatnn$; a partial list includes \cite{negahban2011estimation}, \cite{rohde2011}, \cite{candes2011tight}, and \cite{fan2021shrinkage}, among others. For instance, with a restricted strong convexity assumption on the loss function, \cite{negahban2011estimation} showed that with an appropriate choice of $\lambda$, $\fnorm{\Thetahatnn(\lambda) - \Thetastar}$ is of order $\sqrt{rd / (\kappa n)}$ up to a logarithmic factor, where $\kappa$ is a lower bound of the minimum restricted eigenvalue (\cite{bickel2009}) of the Hessian matrix of the loss function. 
	
	To the best of our knowledge, it remains open whether $\kappa$ is inevitable for statistical guarantees on learning $\Thetastar$. At this point, it is instructive to recall related results for sparse high-dimensional linear regression. \cite{zhang2014lower} showed that under a standard conjecture in computational complexity, the in-sample mean-squared prediction error of any estimator, $\hat \beta_{\mathrm{poly}}$, that can be computed within polynomial time has the following worst-case lower bound: 
    \begin{align}\label{eq:poly_lower}
        \EE \bigg\{ \frac{1}{n} \sum_{i = 1}^n \langle \xvec_i, \hat \beta_{\mathrm{poly}} - \beta ^ * \rangleeuclid ^{2}\bigg\} \gtrsim \frac{(r ^ *) ^ {1 - \delta}\log d}{n\kappa},
    \end{align} 
	where $\delta$ is an arbitrarily small positive scalar. This result demonstrates the indispensable dependence on $\kappa$ for any polynomial-time estimator of $\betastar$, which includes convex estimators like lasso. On the other hand, \cite{bunea2007aggregation} and \cite{raskutti2011minimax} showed that the $L_0$-constrained estimator $\hat \beta_{L_0}$ (also known as the best subset selection estimator), which is defined as 
    \begin{align}  
	    \label{eq:l0_ols}
		\hat \beta_{L_0}(r) \defined \argmin_{\beta \in \R^{\done}, \|\beta\|_0 \leq r} \sum_{i=1}^{n} \left( y_{i} - \langle \xvec_{i}, \beta \rangle_2 \right)^{2}, 
	\end{align}
	satisfies the following $\kappa$-free prediction error bound: 
	\begin{align}
	    \label{eq:l0_upper}
        \EE\bigg\{  \frac{1}{n} \sum_{i = 1} ^ n \langle \xvec_i, \hat \beta_{L_0}(r ^ *) - \betastar \rangleeuclid ^{2}\bigg\} \lesssim \frac{r^{\ast}\log d}{n}. 
    \end{align} 
    This demonstrates the robustness of $\hat\beta_{L_0}$ against collinearity in the design. However, under the general trace regression model, there are currently no $\kappa$-free statistical guarantees for the rank-constrained estimator $\Thetahatlzero$. 

	The first contribution of our work is an in-sample prediction error bound for the rank-constrained least-squares estimator $\Thetahatlzero$ without a restricted strong convexity requirement. We emphasize that this result is much more challenging to achieve than the counterpart result \eqref{eq:l0_upper} for $\hat\beta_{L_0}$ and requires a completely different technical treatment. We shall see in the sequel that the in-sample prediction error of both $\Thetahatlzero$ and $\hat\beta_{L_0}$ boils down to a supremum process of projections of the noise vector $(\epsilon_1, \ldots, \epsilon_n) ^ \top$ onto a family of low-dimensional subspaces. For $\hat\beta_{L_0}$, the family of subspaces is finite; for $\Thetahatlzero$, however, the family of subspaces is a continuous subset of a Stiefel manifold, which is infinite. The main technical challenge we face here is to characterize the complexity of this infinite subspace family. In Theorem \ref{theoremcoveringnumber}, we leverage a real algebraic geometry tool due to \cite{basu2007} to bound the Frobenius-norm-based covering number of this family of subspaces. 
	
	We then investigate a permutation test for the presence of sparse and low-rank signals respectively as applications of the previous results. In the context of hypotheses testing for high-dimensional sparse linear models, \cite{cai2020semisupervised} and \cite{javanmard2020flexible} both consider a debiasing-based test that controls the probability of type I error uniformly over the null parameter space of sparse vectors. There, the sparsity $s ^ *$ of the regression coefficients needs to satisfy $s ^ * = o\{n^{1/2} / \log(d)\}$ for the asymptotic variance of the test statistic to dominate the bias.  By considering a permutation test, we circumvent the challenge of characterizing the asymptotic distribution of a test statistic and accommodate denser alternative parameters.  Moreover, under a mild assumption on the design, we are able to leverage the super-efficiency of the origin, which was instead seen as a challenge in high-dimensional group inference \cite{guo2021group}, to test at a faster rate than $n^{-1/2}$. To the best of our knowledge, this is the first proposal to conduct inference for the presence of low-rank signals.

	\subsection{Organization of the Paper}
	
	In Section \ref{sectionrisk}, we consider a discretization scheme of all possible models in low-rank trace regression and derive the covering number of the corresponding Stiefel sub-manifold that is used to analyze the performance of $\Thetahatlzero$ for in-sample prediction.
	Next, in Section \ref{sectiontesting}, we consider global hypotheses testing in signal-plus-noise models.  We start with a general power analysis for signal-plus-noise models in Section \ref{sectiontestinggeneral}, which we then apply to the sparse high-dimensional linear model and low-rank trace regression model in Sections \ref{sectiontestinglm} and \ref{sectiontestingtrm} respectively.  By leveraging the projection structure of the rank-constrained estimator, in Section \ref{sectiontestingtrmmis}, we demonstrate the robustness of our power analysis to misspecification of the rank. Finally, we analyze the empirical performance of our proposed methodologies in Section \ref{sectionsimulations}.  For the ease of presentation, most of the proofs for Section \ref{sectionrisk} and all of the proofs for Section \ref{sectiontesting} are deferred to Section \ref{sectionproofs}.  The supplement contains additional simulation results for matrix completion.
	
    \section{In-Sample Prediction Risk of the Rank-Constrained Estimator}
	\label{sectionrisk}
	\begin{figure}[t]
	    \centering
	    \includegraphics[scale=.4]{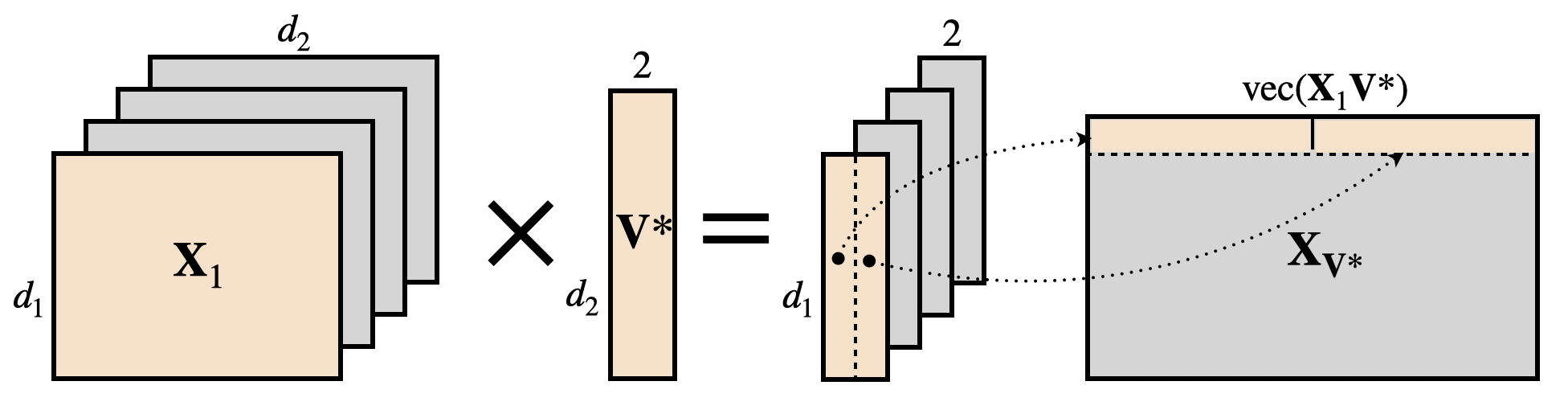}
	    \caption{Illustration of the construction of $\bX_{\bV ^ *}$ with oracle $\bV ^ *$ when $r ^ * = 2$.}
	    \label{fig:reduced_model}
	\end{figure}
	
    Given an estimator $\hat \bTheta$ of $\bTheta ^ *$, define its in-sample prediction risk as 
    \begin{equation}
	    \label{eq:in_sample_pred}
	    \mathcal{R}(\hat{\Thetamat}) \defined \frac{1}{n} \sum_{i=1}^{n} \langle \Xmat_{i}, \hat\bTheta - \Thetastar \ranglehs^{2}.
	\end{equation}
    This section focuses on characterizing the in-sample prediction risk of the rank-constrained estimator $\hat\bTheta_{L_0}$. For any $\mathbf{\Theta} ^ *$ with rank $r^{\ast}$, there exist two matrices $\Umat ^ * \in \R^{\done \times r ^ *}$ and $\Vmat ^ *\in \R^{\dtwo \times r ^ *}$ such that $\mathbf{\Theta} ^ *= \Umat ^ * {\Vmat^ *} ^ \T$.  The existence of $\Umat ^ *$ and $\Vmat ^ *$ is guaranteed, for example, by a singular value decomposition of $\mathbf{\Theta} ^ *$.  Note that this representation is not unique, since for any invertible matrix $\Amat \in \R^{r \times r}$, we have $\mathbf{\Theta} ^ * = \Umat ^ * \Amat \Amat^{-1} {\Vmat ^ *} ^ \top$. Now the trace regression model \eqref{modeltrm} can be represented as 
	\begin{equation}
	    \label{eq:trace_reg_matrix}
		y = \langle \Xmat, \Thetastar \ranglehs + \epsilon 
		= \langle \Xmat, \Umat ^ * {\Vmat ^ *} ^ \T \ranglehs + \epsilon 
		= \langle \Xmat \Vmat ^ *, \Umat ^ *\ranglehs + \epsilon.
	\end{equation}	
	Throughout, for a matrix $\Amat \in \R^{k_{1} \times k_{2}}$, we write $\vec(\Amat) \in \R^{k_{1}k_{2}}$ to denote the vectorization of $\Amat$.  For any $\bV \in \R^{d_2 \times r}$, let $\Xmat_{\Vmat} \in \R^{n \times r\done}$ denote the matrix whose $i$th row is $\vec(\Xmat_{i} \Vmat)$. Figure \ref{fig:reduced_model} illustrates the construction of $\bX_{\bV ^ *}$ when $r ^ * = 2$. Writing $\gammaU \defined \vec(\Umat)$, $\by = (y_1, \ldots, y_n) ^ \top$, and $\bepsilon = (\epsilon_1, \ldots, \epsilon_n) ^ \top$, we then deduce from \eqref{eq:trace_reg_matrix} that
	\begin{align*}
		\by = \Xmat_{\Vmat ^ *} \gammaU + \bepsilon.
	\end{align*}
	Suppose $\bV ^ *$ is known in advance. When $n \gg rd$, $\gammaU$ can be consistently estimated by ordinary least-squares to yield an estimator of $\Thetastar$. Given that ordinary least-squares is projecting $\by$ onto the column space of $\bX_{\bV}$, the rank-constrained least-squares problem \eqref{eq:rank_ols} reduces to finding the optimal $\bV$ so that the resulting $\bX_{\bV}$ captures the most variation of the response $\by$. This motivates our initial step to analyze the in-sample prediction risk. For any $\bV \in \R ^ {d_2 \times r}$, define the projection matrix $\bP_{\bV} := \bX_{\bV}(\bX_{\bV} ^ \top  \bX_{\bV}) ^ {-1}\bX_{\bV} ^ \top$. The following lemma shows that the in-sample prediction risk of $\hat\bTheta_{L_0}$ can be bounded by the supremum of projections of the noise vector $\bepsilon$ onto the column space of $\bX_{\bV}$. 
    \begin{lemma}
        \label{lem1}
        Consider the model in equation \eqref{modeltrm}.  If $r \geq \rstar$, then, the rank-constrained least-squares estimator satisfies
    	\begin{equation}
    	    \label{eq:sup_proj}
% 			n^{-1} \sum_{i=1}^{n} \langle \Xmat_{i}, \Thetahatlzero(r ^ *) - \Thetastar \ranglehs^{2}
            \mathcal{R}(\Thetahatlzero(r)) 
			\leq \frac{4}{n} \sup_{\Vmat \in \R^{\dtwo \times 2r}} \Vert \projV \bepsilon \Vert_{2}^{2}.
		\end{equation}    
    \end{lemma}
    
    Lemma \ref{lem1} suggests that the complexity of the set $\projset=\projset(2r) \defined \{\projV\}_{\Vmat \in \R^{\dtwo \times (2r)}}$ determines the in-sample prediction risk of $\hat\bTheta_{L_0}(r)$. Note that the number of columns of $\bV$ is $2r$ instead of $r ^ *$, which is due to the fact that the maximum rank of $(\hat\bTheta_{L_0}(r) - \Thetastar)$ is $r + \rstar \leq 2r$. To quantify this complexity, we consider the metric space $(\projset, \Vert \cdot \Verths)$.  We say that $\netepsilon \subseteq \projset$ is an $\varepsilon$-net of $\projset$ if for any $\bP \in \projset$, there exists a $\tilde{\bP} \in \netepsilon$ such that $\Vert \bP - \tilde{\bP} \Verths \leq \varepsilon$.  We define the covering number, $\covnumepsilon(\projset)$, as the minimum cardinality of an $\varepsilon$-net of $\projset$.  The following theorem leverages a result \cite{basu2007} from real algebraic geometry to bound $\covnumepsilon(\projset)$. To the best of our knowledge, this tool is new to statistical analyses in high-dimensions.  To highlight the power of the tool, we give the proof immediately after the statement of the theorem.
	
	\begin{theorem}\label{theoremcoveringnumber}
		For any $\varepsilon < 1$ and any positive integer $r$, we have that 
		\begin{align*}
			\covnumepsilon (\projset(r)) \leq 2^{r\done} \left\{ \frac{12r\done n^{3}}{\varepsilon} \right\}^{r\dtwo + 1}.
		\end{align*}
	\end{theorem}
    \begin{proof}[Proof of Theorem \ref{theoremcoveringnumber}]	
% 	To simplify notation, we drop the superscript $ ^ *$ in $r ^ *$ for convenience within this proof.
	For an integer $k$, let $[k] \defined \{1, \dots, k\}$.  Now, for $\sset \subseteq [r\done]$, write $\Xmat_{\vmat, \sset}$ to denote the $\R^{n \times |\sset|}$ submatrix of $\Xmat_{\Vmat}$ with the columns indexed by $\sset$.  Then, for any fixed $\sset \subseteq [r\done]$, define the collection of projection matrices $\cP_{\sset}$ as 
	\begin{align*}
		\projsets \defined \{\Xmat_{\vmat, \sset} (\Xmat_{\vmat, \sset}^\T \Xmat_{\vmat, \sset})^{-1} \Xmat_{\vmat, \sset}^\T : \Vmat \in \R^{\dtwo \times r}, \det(\Xmat_{\vmat, \sset}^\T \Xmat_{\vmat, \sset}) \neq 0\}.
	\end{align*}
	Note that 
	\begin{align*}
		\projset \defined \{ \projV \}_{\Vmat \in \R^{\dtwo \times r}} \subseteq \bigcup_{\sset \subseteq [r\done]} \projsets.
	\end{align*}
	
	To further simplify notation, throughout this proof, we identify the matrix $\bV \in \RR^{\dtwo \times r}$ with the vector $\bv \in \RR^{r\dtwo}$ by viewing $\bv$ as the vectorization of $\bV$.  Fix $\{\bX_i\}_{i \in [n]}$ and $i, j \in [n]$ and consider the map
		\begin{align*}
			& \Phi_{ij}: \R ^ {r\dtwo} \backslash \{\vvec : \det ( \Xmat_{\vmat, \sset}^\T \Xmat_{\vmat, \sset} ) = 0 \} \to [-1, 1], \\
			& ~\qquad \bv \mapsto (\bP_{\bV, \cS})_{i,j}= \{\bX_{\bV, \cS}(\bX_{\bV, \cS} ^ {\top} \bX_{\bV, \cS}) ^ {-1} \bX_{\bV, \cS} ^{\top}\}_{ij}. 
		\end{align*}
		We claim that $\Phi_{ij}$ is a rational function of polynomials of order at most $2|\sset|$.  To see this, for any invertible matrix $\bA \in \R^{p \times p}$ and $u, t \in [p]$, the $(u, t)$ entry of the adjugate of $\bA$ is given by $\adj(\bA)_{ut} \defined (-1)^{u + t} \det(\bA_{-u, -t})$.  Then, by Cramer's rule, 
		\begin{align*}
			\bA_{i,j}^{-1} = (\det (\bA))^{-1} \adj(\bA)_{ij}.
		\end{align*}
		Given that each entry of $\Xmat_{\Vmat}^\T \Xmat_{\Vmat} \in \R^{|\sset| \times |\sset|}$ is a quadratic function with respect to $\vvec$, it follows that $(\det (\Xmat_{\Vmat}^\T \Xmat_{\Vmat}))^{-1}$ is a polynomial of order at most $2|\sset|$ and $\adj(\bA)_{ut}$ is a polynomial of order at most $2|\sset| - 2$.  Hence, each entry of $\projV$ is a rational function of polynomials of order at most $2|\sset|$.  Denote the $(i,j)$ entry of $\projV$ by $\Phi_{i,j}(\vvec)$, which has representation $\Phi_{i,j}(\vvec) = F_{i,j}(\vvec) / \gamma(\vvec)$ for polynomials $F_{i,j}(\vvec)$ and $\gamma(\vvec)$ in the domain of $\Phi_{i,j}$.
		
		Now for any $\varepsilon > 0$, consider a monotonically increasing sequence $-1 = s_1 < \ldots < s_m = 1$ such that $m = \lceil{2 / \varepsilon}\rceil + 1$ and $|s_{t + 1} - s_t| \le \varepsilon$ for any $t \in [m - 1]$. Consider the level sets: $\cC_{ijt}\defined \{\bv \in \RR ^ {2r\dtwo}: (F_{i,j}(\bv) - \gamma(\bv)s_{t}) = 0\}, t \in [m]$. Note that these level sets partition the entire $\RR ^ {r\dtwo}$ into multiple connected components, within each of which any two points $\bv_1, \bv_2$ satisfy $|\Phi_{ij}(\bv_1) - \Phi_{ij}(\bv_2)| \le \varepsilon$. Consider the union of all such level sets over $i, j, t$: 
		\begin{align*}
			\cC \defined \bigcup_{i, j \in [n], t \in [m]} \cC_{ijt} = \biggl\{\bv \in \RR^ {r\dtwo}: \prod_{i,j \in [n], t \in [m]} (F_{i,j}(\bv) - \gamma(\bv)s_{t}) = 0\biggr\}. 
		\end{align*}
		For any two points $\bv_1$ and $\bv_2$ in a single connected component of the complement, $\cC ^ \C$, $|\Phi_{ij}(\bv_1) - \Phi_{ij}(\bv_2)| \le \varepsilon$ for all $(i, j) \in [n] \times [n]$. Therefore, $\fnorm{\bP_{\bv_1, \cS} - \bP_{\bv_2, \cS}} \le n\varepsilon$. This implies that $\covnum_{n\varepsilon}(\cP_{\cS})$ is bounded by the number of connected components of $\cC ^ \C$. Define 
		\begin{align*}
			\Phi: \RR ^ {r\dtwo + 1} \to \RR, (v_0, \bv ^ {\top}) ^ {\top} \mapsto \biggl\{ \gamma(\bv) \times \prod_{i,j \in [n], t \in [m]} v_0(F_{i,j}(\bv) - \gamma(\bv)s_{t})\biggr\} - 1. 
		\end{align*}
		We have that $\Phi ^ {-1}(0)$ shares the same number of connected components as $\cC ^ \C$. By Theorem 7.23 of \cite{basu2007}, the number of connected components of $\Phi ^ {-1}(0)$, which is the $0$th Betti number of $\Phi ^ {-1}(0)$, is bounded by $\{(4|\sset| - 1) n ^2 m\} ^ {r\dtwo + 1}$. Therefore, for any $\varepsilon < 1$, 
		\begin{align*}
			\covnum_{n\varepsilon}(\cP) \le \biggl\{(4|\sset| - 1) n ^2 \frac{3}{\varepsilon}\biggr\} ^ {r\dtwo + 1}.
		\end{align*}
		Then we deduce that
		\begin{align*}
			\covnumepsilon (\projsets) \leq \left\{ \frac{3(4|\sset| - 1)n^{3}}{\varepsilon} \right\}^{r\dtwo + 1}.
		\end{align*}

	    Finally, since $\projset \subseteq \bigcup_{\sset \subseteq [r\done]} \projsets$, we have
		\begin{align*}
			\covnumepsilon(\projset) 
			\leq \sum_{\sset \subseteq [r\done]} \covnumepsilon(\projsets) 
			\leq \sum_{\sset \subseteq [r\done]} \left\{ \frac{3(4|\sset| - 1)n^{3}}{\varepsilon} \right\}^{r\dtwo + 1}
			\leq 2^{r\done} \left\{ \frac{12r\done n^{3}}{\varepsilon} \right\}^{r\dtwo + 1}, 
		\end{align*}
		which concludes the proof.
	\end{proof}

	To bound the in-sample prediction risk with high-probability, we need the following mild assumption, which is standard in high-dimensional models.  In order to state our assumption, we first define sub-Gaussian random variables.
	
	\begin{definition}
	    For a random variable $\xi$ valued in $\R$, define the $\psi_{2}$-norm of $\xi$, denoted $\Vert \xi \Vert_{\psi_{2}}$, as 
	    \begin{align*}
	        \Vert \xi \Vert_{\psi_{2}} \defined \inf_{t>0}\{\e \exp(t^{-2} \xi^{2}) \leq 2 \}.
	    \end{align*}
	
        Then, define the family of sub-Gaussian random variables with parameter $K$ as
        \[
            % \sg(\sgparam) \defined \big\{\xi: \inf_{t>0}\{\e \exp(t^{-2} \xi^{2}) \leq 2 \} \le K \big\}.
            \sg(\sgparam) \defined \big\{\xi: \Vert \xi \Vert_{\psi_{2}} \le K \big\}.
        \]
		More generally, for $p$-dimensional real-valued random vectors, we define the sub-Gaussian family with parameter $K$ as
		\begin{align*}
            \sg_p(\sgparam) \defined \Big\{\xivec: \sup_{\vvec \in \RR ^ p, \|\vvec\|_2 = 1} \Vert \langle \xivec, \vvec \rangleeuclid \Vert_{\psi_{2}} \le K\Big\}.
		\end{align*}
	\end{definition}

	\begin{assumption}\label{assumptionepsilon}
		The noise $\epsilon \in \sg(\sgparamepsilon)$ with mean zero and variance $\sigmaepsilonsq$ and is independent of $\Xmat$.
	\end{assumption}
	
	Now, we can state our main result for $\Thetahatlzero$.
	
	\begin{theorem}\label{theorempredictionrisk}
		Suppose we have $n$ observations $(\Xmat_i, y_i)_{i \in [n]}$ from model \eqref{modeltrm} with $(\epsilon_i)_{i \in [n]}$ independent.  Under Assumption \ref{assumptionepsilon}, if $\rhat \geq \rstar$, then there exist $\cone, \ctwo > 0$ depending on $\sigmaepsilonsq$ and $\sgparamepsilon$ such that
		\begin{align*}
		  %  n^{-1} \sum_{i=1}^{n} \langle \Xmat_{i}, \Thetahatlzero - \Thetastar \ranglehs^{2}
		    \mathcal{R}(\Thetahatlzero(r))
		    \leq \cone \frac{rd\log(n)}{n}
		\end{align*}
		with probability at least $1 - 4 \exp(-\ctwo rd\log(n))$.
	\end{theorem}
	
	Theorem \ref{theorempredictionrisk} should be compared with the results of \cite{rohde2011} and \cite{koltchinskii2011}, who proved bounds on in-sample prediction for the estimator $\Thetahatnn$.  Up to logarithmic factors, both $\Thetahatlzero$ and $\Thetahatnn$ achieve the same in-sample prediction risk; however, the crucial difference between our result and the existing results is the assumption, or lack thereof, on the design matrices, $\Xmat_{i}$.  The estimator $\Thetahatnn$, much like the lasso estimator for linear models, requires a restricted eigenvalue type assumption in order to enjoy near optimal rates of in-sample prediction risk.  By comparison, Theorem \ref{theorempredictionrisk} imposes no such requirement.
	
	In the theorem above, we assume that the tuning parameter, $r$, exceeds the true rank, $\rstar$.  The following corollary extends the result to the setting where $r < \rstar$. Moreover, it accommodates potential model misspecification by allowing data to come from a general signal-plus-noise model. 

	\begin{corollarytheorem}\label{corollarypredictionrisk}
		Consider $n$ observations $(y_i)_{i \in [n]}$ from a signal-plus-noise model
		\begin{align*}
			y = f(x) + \epsilon, 
		\end{align*}
		where $x \in \mathscr{X}$. 
% 		For $\rhat > 0$, define 
% 		\begin{align*}
% 			\Thetahatlzero(\rhat) \defined \argmin_{\Thetamat \in \R^{\done \times \dtwo}, \rank(\Thetamat) \leq \rhat} \sum_{i=1}^{n} (y_{i} - \langle \Xmat_{i}, \Thetamat \ranglehs)^{2}.
% 		\end{align*}
		Assume that $\{\epsilon_{i}\}_{i \in [n]}$ are independent and identically distributed sub-Gaussian random variables with parameter $\sgparamepsilon$.  If $\Xmat_{i}$ is independent of $\epsilon_{i}$, then there exist $\cone, \ctwo > 0$ depending on $\sigmaepsilonsq$ and $\sgparamepsilon$ such that 
		\begin{align*}
% 			\mathcal{R}(\Thetahatlzero)
% 			\leq \Big\{ 
% 			\Big[ \min_{\Thetamat \in \R^{\done \times \dtwo}, \rank(\Thetamat) \leq \rhat} \mathcal{R}(\Thetamat)\Big]^{1/2}
% 			+ \Big[\cone \frac{rd\log(n)}{n} \Big]^{1/2} \Big\}^{2}
            &\frac{1}{n}\sum_{i=1}^{n} (f(x_{i}) - \langle \Xmat_{i}, \Thetahatlzero\ranglehs)^{2} \\ 
			&\phantom{\leq} \leq \Big\{ 
			\Big[ \min_{\substack{\Thetamat \in \R^{\done \times \dtwo}, \\\rank(\Thetamat) \leq \rhat}} \frac{1}{n} \sum_{i=1}^{n} (f(x_{i}) - \langle \Xmat_{i}, \Thetamat \ranglehs)^{2} \Big]^{1/2}
			+ \Big[\cone \frac{rd\log(n)}{n} \Big] ^ {1/2} \Big\}^{2}
		\end{align*}
		with probability at least $1 - 4 \exp(-\ctwo rd\log(n))$.
	\end{corollarytheorem}
	
	Compared with Theorem \ref{theorempredictionrisk}, Corollary \ref{corollarypredictionrisk} has an additional term denoting the best attainable risk using a rank-$r$ approximation; when $r \geq \rstar$ and $f(x_i) = \langle \bX_i, \Thetamat ^ *\ranglehs$, the best approximation error is zero and we recover Theorem \ref{theorempredictionrisk}.
	
	\section{Testing in Signal-Plus-Noise Models}\label{sectiontesting}
	
	\subsection{A General Power Analysis of a Permutation Test}\label{sectiontestinggeneral}
	Consider the following general signal-plus-noise model:
	\begin{align}\label{modelsignalplusnoise}
		y = f(x) + \epsilon \defined \e(y | x) + \epsilon,
	\end{align}
	where $x$ is a random covariate valued in $\mathscr{X}$. 
	We are interested in the hypotheses testing problem 
	\begin{align}\label{hypothesesnull}
		H_{0}: f \equiv 0 &&& H_{1}: f \in \mathscr{F}, f\not\equiv 0
	\end{align}
	for some prespecified function class $\mathscr{F}$.  We discuss some examples of $\mathscr{F}$ in Sections \ref{sectiontestinglm} -- \ref{sectiontestingtrm}.  To test these hypotheses, we consider a flexible permutation test.  We note that the application of permutation tests to signal-plus-noise models is not novel.  For example, the seminal work of \cite{freedman1983} proposed a permutation test for a collection of covariates in a low-dimensional linear model.  However, to the best of our knowledge, the theory of permutation tests for high-dimensional models has not been explored, particularly the power of permutation tests.  Under mild assumptions, such as exchangability of $(\epsilon_{i})_{i\in[n]}$ given $(x_{i})_{i\in[n]}$,  it is easy to derive a test statistic that controls type I error under the permutation null hypothesis.  However, the sparsity rate enters in the power under the alternative.  In Theorem \ref{theoremtestingalternative} below, we characterize explicitly this dependence of power on the sparsity. 
	
	Before presenting the test statistic, we need to establish some notation and facts from enumerative combinatorics regarding permutations that are used throughout the section. Suppose we have $n$ independent observations $(x_i, y_i)_{i \in [n]}$ from model \eqref{modelsignalplusnoise}. Let $\Pi = \Pi_{n}$ denote the set of all permutations over $[n]$.  For a given $\pi \in \Pi$, we write $\fpi(x_{i}) \defined \e(y_{\pi(i)} | x_{i})$, noting that if $\pi(i) \neq i$, then $\fpi(x_{i}) = 0$.  We define $\pizero$ to be the identity permutation on $[n]$.  Moreover, a permutation $\pi \in \Pi$ induces a partition of $[n]$ into $K(\pi)$ cycles, where a cycle is an index subset $\{i_{1}, \dots, i_{k}\}$ such that $\pi(i_{j}) = i_{j+1}$ for $j \in [k-1]$ and $\pi(i_{k}) = i_1$. Let
	\begin{align*}
		\Pitilde \defined \Pitilde_{n} = \{\pi \in \Pi : K(\pi) \leq \log^{2}(n)\}.
	\end{align*}
	We make the following assumptions.
	\begin{assumption}\label{assumptiontestingmoments}
		The mean $f(x)$ satisfies that $\e [f(x)] = 0$, $\e [f^{2}(x)] = \sigmamusq$, and $\e [f^{4}(x)] < \infty$.  The error $\epsilon$ satisfies that $\e (\epsilon) = 0$ and $\e(\epsilon ^ {2}) = \sigmaepsilonsq$ and is independent of $x$.
	\end{assumption}

	\begin{assumptionstar}\label{assumptiontestingmeansubgaussian}
		The mean $f(x)$ satisfies that $\var(f^{2}(x)) \leq \kol \sigmamu^{4}$ for some constant $\kol > 0$.
	\end{assumptionstar}
	
	\begin{assumption}\label{assumptiontestingconsistency}
		For a fixed $\delta > 0$, there exists an estimator $\fhat : \mathscr{X} \times (\mathscr{X} \times \R)^{n} \times \Pi \to \R$ and a sequence $\ell_{n}$ (possibly depending on $\delta$) such that 
		\begin{enumerate}[label=(\roman*)]
			\item the estimator $\hat{f}$ is equivariant in the sense that for any $\pi \in \Pi$,
			\begin{align*}
				\fhat(x_{i} ; (x_{j}, y_{j})_{j=1}^{n} ; \pi) = \fhat(x_{i} ; (x_{j}, y_{\pi(j)})_{j=1}^{n} ; \pizero).
			\end{align*}
			
			\item for $n$ sufficiently large, 
			\begin{align*}
				\min_{\pi \in \Pitilde \cup \{\pizero\}} \p\Big\{ \sum_{i=1}^{n} [\fhat(x_{i} ; (x_{j}, y_{j})_{j=1}^{n} ; \pi) - \fpi(x_{i})]^{2} \leq \ell_{n} \Big\} \geq 1 - \delta.
			\end{align*}
		\end{enumerate}
	\end{assumption}
		
	Temporarily fix $\pi \in \Pi$. For convenience, let $\fhatpi(\cdot) \defined \fhat(\cdot ; (x_{j}, y_{j})_{j=1}^{n} ; \pi)$.  Now, given an estimation procedure $\fhat : \mathscr{X} \times (\mathscr{X} \times \R)^{n} \times \Pi \to \R$ satisfying Assumption \ref{assumptiontestingconsistency}, we define $\Lambdapi$ as
	\begin{align*}
		\Lambdapi \defined \Lambdapi(\fhat) = \sum_{i=1}^{n} [\fhatpi(x_{i})]^{2}.
	\end{align*}
	Then, our p-value is given by
	\begin{align*}
		\phi(\fhat) \defined \frac{1}{|\Pi|} \sum_{\pi \in \Pi} \indic{\Lambdapizero(\fhat) \leq \Lambdapi(\fhat)}.
	\end{align*}
	
	Assumption \ref{assumptiontestingmoments} is standard, imposing an independence assumption and some moment conditions on the model.  The requirement for the existence of the fourth moment of $f(x_{i})$ is to ensure the concentration of $\Vert \muvec \Vert_{2}^{2}$ around $n\sigmamusq$.  The next assumption, \ref{assumptiontestingmeansubgaussian}, is a technical condition that allows for a faster concentration of $\Vert \muvec \Vert_{2}^{2}$; the faster concentration yields a sharper rate in the contiguous alternative.  For example, if the $f(x_{i})$ are Gaussian, then Assumption \ref{assumptiontestingmeansubgaussian} is satisfied with $\kol = 3$.  
	
	Assumption \ref{assumptiontestingconsistency} is a very natural assumption, albeit technical.  For the first part, the symmetry in the estimation procedure, $\fhat(\cdot)$, implies that $\Lambdapi$ is identically distributed under the null hypothesis for all $\pi \in \Pi$.  For the second half, we assume that $\fhatpi(\cdot)$ is a consistent estimator of $\fpi(\cdot)$ for any $\pi \in \Pi$ at a rate slightly faster than $\ell_{n}$.  In particular, for most $\pi \in \Pi$, the estimator $\fhatpi(\cdot)$ approximates the zero function.  To see this, let $C_{1}, \dots, C_{K(\pi)}$ denote a fixed representation of the $K(\pi)$ cycles, for example as expressed in \emph{standard representation} (see \cite{stanley2011} for a formal definition).  For $j \in [K(\pi)]$ and $i \in C_{j}$, let $m(i)$ denote the index of $i$ in $C_{j}$.  Then, define the sets $\Asetpione$, $\Asetpitwo$, and $\Asetpithree$ as follows:
	\begin{align*}
		&\Asetpione \defined \bigcup_{j \in [K(\pi)]} \{i \in C_{j} : m(i) \text{ is odd and } m(i) \neq |C_{j}|\}, \\ 
		&\Asetpitwo \defined \bigcup_{j \in [K(\pi)]} \{i \in C_{j} : m(i) \text{ is even}\}, \\ 
		&\Asetpithree \defined [n] \cap (\Asetpione \cup \Asetpitwo)^\C.
	\end{align*}
	For example, for the permutation expressed by the cycles  $(4321)(765)(8)$, we have $\Asetpione = \{2, 4, 7\}$, $\Asetpitwo = \{1, 3, 6\}$, and $\Asetpithree = \{5, 8\}$.  Intuitively, $\Asetpione$ and $\Asetpitwo$ are two sets of observations such that, within each set, the covariates and responses are mutually independent.  Therefore, for $i \in \Asetpione \cup \Asetpitwo$, we have that $\fpi(x_{i}) = 0$.  The other set, $\Asetpithree$, are the remaining observations.  To bound the error in the remaining observations, we note that $|\Asetpithree| \le K(\pi)$ and that $\e K(\pi) / \log(n) \to 1$ as $n \to \infty$, where the expectation is with respect to the uniform probability measure over $\Pi$ (cf. \cite{stanley2011}).  Now, by Markov's inequality, for large values of $n$, 
	\begin{align*}
		\frac{| \Pitilde^\C |}{| \Pi |} = \p(K(\pi) > \log^{2}(n)) \leq \frac{2}{\log(n)} \to 0.  
	\end{align*}
	This leads to the following lemma, which asserts that, for $\pi \in \Pitilde$, the conditional mean function, $\fpi(\cdot)$, is approximately zero. 
	
	\begin{lemma}\label{lemmapermutationmeanremainder}
		Under Assumption \ref{assumptiontestingmoments}, for any $\delta > 0$, there exists $\cone > 0$ depending on $\delta$ and $\E f ^ 4(x)$ such that
		\begin{align*}
			\min_{\pi \in \Pitilde} \p \Big\{\sum_{i \in \Asetpithree} [\fpi(x_{i})]^{2} \leq \log^{2}(n) \sigmamusq + \cone \log(n) \Big\} \geq 1 - \delta.
		\end{align*}
	\end{lemma}
	
	As an immediate consequence of Lemma \ref{lemmapermutationmeanremainder}, we have the following corollary, which yields an alternative way to check the second half of Assumption \ref{assumptiontestingconsistency}.
	
	\begin{corollarylemma}\label{corollarypermutationmeanremainder}
         Suppose that, for a fixed $\delta > 0$, there exists an estimator $\fhat : \mathscr{X} \times (\mathscr{X} \times \R)^{n} \times \Pi \to \R$ and a sequence $\ell_{n}$ with $\log^{2}(n) = o(\ell_{n})$ such that for $n$ sufficiently large, 
		\begin{align*}
			\p\Big\{\sum_{i=1}^{n} [\fhat(x_{i} ; (x_{j}, y_{j})_{j=1}^{n} ; \pizero) - \fpizero(x_{i})]^{2} \leq \ell_{n} \Big\} \geq 1 - \delta
		\end{align*}
		and 
		\begin{align*}
			\min_{\pi \in \Pitilde} \p\Big\{\sum_{i=1}^{n} [\fhat(x_{i} ; (x_{j}, y_{j})_{j=1}^{n} ; \pi)]^{2} \leq \ell_{n} \Big\} \geq 1 - \delta.
		\end{align*}
		Then, under Assumption \ref{assumptiontestingmoments}, for $n$ sufficiently large, 
		\begin{align*}
			\min_{\pi \in \Pitilde \cup \{\pizero\}} \p\Big\{\sum_{i=1}^{n} [\fhat(x_{i} ; (x_{j}, y_{j})_{j=1}^{n} ; \pi) - \fpi(x_{i})]^{2} \leq \ell_{n} \Big\} \geq 1 - \delta.
		\end{align*}
	\end{corollarylemma}
	
	We can now state our first result that $\phi$ controls the type I error.
	
	\begin{theorem}\label{theoremtestingnull}
		Consider model \eqref{modelsignalplusnoise} with the hypotheses testing problem in \eqref{hypothesesnull}.  Under Assumptions \ref{assumptiontestingmoments},   \ref{assumptiontestingconsistency} and the null hypothesis, we have that
		\begin{align*}
			\limsup_{n \to \infty} \phzero (\phi(\fhat) \leq \alpha) \leq \alpha.
		\end{align*}
	\end{theorem}
	
	To analyze the power of the test, we consider two contiguous hypotheses testing problems, depending on whether we impose Assumption  \ref{assumptiontestingmeansubgaussian}.  First, consider
	\begin{align}\label{hypothesesalternative}
		H_{0}: f \equiv 0
		\vs
		H_{1}: f \not\equiv 0, f \in \mathscr{F}, \sigmamusq = h\Big(\frac{1}{\sqrt{n}} + \frac{\ell_{n}}{n} \Big), 
	\end{align}
	where $h > 0$ controls the signal strength under $H_1$.  Then, we have the following theorem that 
	\begin{theorem}\label{theoremtestingalternative}
		Consider model \eqref{modelsignalplusnoise} with the hypotheses testing problem \eqref{hypothesesalternative}.  Suppose Assumptions \ref{assumptiontestingmoments} and \ref{assumptiontestingconsistency} hold with $\delta < \alpha (1 - \alpha) / 4$.  If $h$ is sufficiently large (not depending on $n$), then, under the alternative hypothesis in \eqref{hypothesesalternative}, 
		\begin{align*}
			\liminf_{n \to \infty} \phone (\phi(\fhat) \leq \alpha) > \alpha.
		\end{align*}
	\end{theorem}
	
	If we further assume \ref{assumptiontestingmeansubgaussian}, we consider the following hypotheses
	\begin{align}\label{hypothesesalternativesg}
		H_{0}: f \equiv 0
		\vs
		H_{1}: f \not\equiv 0, f \in \mathscr{F}, \sigmamusq = h\frac{\ell_{n}}{n}.
	\end{align}
	We have the following corollary.
	
	\begin{corollarytheorem}\label{corollarysgmean}
		Consider model \eqref{modelsignalplusnoise} with the hypotheses testing problem in \eqref{hypothesesalternativesg}.  Under Assumptions \ref{assumptiontestingmoments}, \ref{assumptiontestingmeansubgaussian}, \ref{assumptiontestingconsistency} with $\delta < \alpha (1 - \alpha) / 4$ and the alternative hypothesis in \eqref{hypothesesalternativesg}, if $h$ is sufficiently large (not depending on $n$), then we have that
		\begin{align*}
			\liminf_{n \to \infty} \phone (\phi(\fhat) \leq \alpha) > \alpha.
		\end{align*}
	\end{corollarytheorem}
	
	Comparing the hypotheses in equations \eqref{hypothesesalternative} and \eqref{hypothesesalternativesg}, one can see that Assumption \ref{assumptiontestingmeansubgaussian} allows for testing at a rate faster than $n^{-1/2}$.  In view of Corollary \ref{corollarypermutationmeanremainder}, we emphasize that the bottleneck of testing power under Assumption \ref{assumptiontestingmeansubgaussian} is the rate at which we can predict the conditional mean given the permuted covariates.  
	
	In the following two sections, we apply the established theory to develop valid permutation tests that can distinguish alternatives at a fast rate for sparse high-dimensional models and low-rank models respectively. 

	\subsection{Sparse High-Dimensional Linear Model}\label{sectiontestinglm}
	In this section, we focus on model \eqref{modellm} and consider $\mathscr{F} \defined \{ f: \R^{p} \to \R \mid f(\xvec) = \langle \xvec, \betastar \rangleeuclid, \betastar \in \R^{p}, \Vert \betastar \Vert_{0} \leq \sstar\}$.
	
	\begin{assumption}\label{assumptiontestinglm}
		The covariate vector $\xvec \in \sg_p(\sgparamx)$ has mean zero and variance $\mathbf{\Sigma}$, and the error $\epsilon \in \sg(\sgparamepsilon)$ has mean zero and variance $\sigmaepsilonsq$. Moreover, $\xvec$ is independent of $\epsilon$.  
	\end{assumption}
	
	\begin{assumptionstar}\label{assumptiontestingkoltchinskii}
		There exists $\kol > 0$ such that 
		\begin{align*}
			\Vert \langle \xvec, \vvec \rangleeuclid \Vert_{\psi_{2}}^{2} \leq \kol \e\big(\langle \xvec, \vvec \rangleeuclid^{2}\big) 
		\end{align*}
		for all $\vvec \in \R^{p}$.  
	\end{assumptionstar} 
	
	The first half of Assumption \ref{assumptiontestinglm} is mild, assuming a random sub-Gaussian design framework that is standard in the high-dimensional setting.  In particular, it implies Assumption \ref{assumptiontestingmoments}. Assumption \ref{assumptiontestingkoltchinskii} is a technical assumption that is used in the literature for concentration of the sample covariance matrix.  For example, see Definition 2 of \cite{koltchinskii2017} or Theorem 4.7.1 of \cite{vershynin2018}.  In particular, it implies Assumption \ref{assumptiontestingmeansubgaussian}, and, as an example, the Gaussian distribution satisfies this assumption.
	
	Then, the pairs of contiguous testing problems that we consider are 
	\begin{align}\label{hypothesesalternativelm}
		H_{0}: \betastar = \zerovec_{p}
		\vs
		H_{1}: \Vert \betastar \Vert_{0} = \sstar > 0, (\betastar)^\T \mathbf{\Sigma} \betastar = h\Big(\frac{1}{\sqrt{n}} + \frac{\sstar \log(p)}{n}\Big)
	\end{align}
	and, under Assumption \ref{assumptiontestingkoltchinskii},
	\begin{align}\label{hypothesesalternativelmsg}
		H_{0}: \betastar = \zerovec_{p}
		\vs
		H_{1}: \Vert \betastar \Vert_{0} = \sstar > 0, (\betastar)^\T \mathbf{\Sigma} \betastar = h \frac{\sstar \log(p)}{n}.
	\end{align}
		
	Now, for any $\pi \in \Pi$, we define the lasso estimator as 
	\begin{align*}
		\fhatlasso(\xvec_{i} ; (\xvec_{j}, y_{j})_{j=1}^{n} ; \pi) \defined \langle \xvec_{i}, \betahatlassopi \rangleeuclid,
	\end{align*}
	where 
	\begin{equation}
	    \label{eq:lasso}
		\betahatlassopi \defined \argmin_{\beta \in \R^{p}} \Big\{\frac{1}{n} \sum_{i=1}^{n} (y_{\pi(i)} - \langle \xvec_{i}, \beta \rangleeuclid)^{2} + \lambda \Vert \beta \Vert_{1} \Big\}.
	\end{equation}
	Then, we have the following result for the lasso estimator.
	
	\begin{theorem}\label{theoremtestinglmlasso}
		Consider model \eqref{modellm}.  Suppose that Assumption  \ref{assumptiontestinglm} holds with $0 < \lambdamin(\mathbf{\Sigma}) \leq \lambdamax(\mathbf{\Sigma}) < \infty$.  Then for a fixed value of $\delta > 0$, there exists sequence $\ell_{n} = O(sn\lambda^{2})$ such that the lasso estimator $\fhatlasso$ satisfies (i) and (ii) of Assumption \ref{assumptiontestingconsistency}, provided that the tuning parameter $\lambda$ in \eqref{eq:lasso} satisfies $\lambda \geq 2\lambdazero$ and $\lambda \asymp \lambdazero$, where
		\begin{align*}
			\lambda_{0} \geq 
			\cone \sqrt{\sgparamx(\sgparammu + \sgparamepsilon)\frac{\log(6 / \delta) + \log(p)}{n}}
		\end{align*}
		for some universal constant $\cone > 0$.
	\end{theorem}
	
	Given Theorem \ref{theoremtestinglmlasso}, applying Theorems \ref{theoremtestingnull}, \ref{theoremtestingalternative} and Corollary \ref{corollarysgmean} yields the following corollary on the asymptotic validity of the permutation test based on $\fhatlasso$.
	
	\begin{corollarytheorem}\label{corollarytestinglmlasso}
		Under the assumptions of Theorem \ref{theoremtestinglmlasso}, 
		\begin{align*}
		    \limsup_{n \to \infty} \phzero(\phi(\fhatlasso) \leq \alpha) \leq \alpha
		\end{align*}
		In addition, if $h$ is sufficiently large (not depending on $n$), then
		\begin{align*}
			\liminf_{n \to \infty} \phone(\phi(\fhatlasso) \leq \alpha) > \alpha
		\end{align*}
		 for the hypotheses testing problem in equation \eqref{hypothesesalternativelm} and also for the hypotheses in equation \eqref{hypothesesalternativelmsg} if Assumption \ref{assumptiontestingkoltchinskii} holds.
	\end{corollarytheorem}
	
	Similarly, we define the $L_{0}$ estimator as
	\begin{align*}
		\fhatlzero(\xvec_{i} ; (\xvec_{j}, y_{j})_{j=1}^{n} ; \pi) \defined \langle \xvec_{i}, \betahatlzeropi \rangleeuclid, 
	\end{align*}
	where 
	\begin{equation}
	    \label{eq:lzero}
		\betahatlzeropi \defined \argmin_{\beta \in \R^{p}, \Vert \beta \Vert \leq \shat} \sum_{i=1}^{n} (y_{\pi(i)} - \langle \xvec_{i}, \beta \rangleeuclid)^{2}.
	\end{equation}
	Now, the following theorem is the analogue of Theorem \ref{theoremtestinglmlasso} for the $L_{0}$ estimator.
	
	\begin{theorem}\label{theoremtestinglmlzero}
		Consider model \eqref{modellm}. Suppose that $\shat \asymp \sstar$ with $\shat \geq \sstar$ in \eqref{eq:lzero}.  Then under Assumption \ref{assumptiontestinglm}, for a fixed value of $\delta > 0$, the $L_{0}$ estimator, $\fhatlzero(\cdot)$, satisfies Assumption \ref{assumptiontestingconsistency} with $\ell_{n} = O( s \log(p) + \log(1 / \delta))$.
	\end{theorem}
	
 	 It is worth emphasis that Theorem \ref{theoremtestinglmlzero} does not require the minimum eigenvalue of $\bSigma$ to be well bounded  from below as in Theroem \ref{theoremtestinglmlasso}. This demonstrates the robustness of the $L_0$ estimator against collinearity of the covariates when it is compared with the lasso. In the following, we establish the asymptotic validity of the permutation test based on $\hat f_{L_0}$, again without any requirement on $\lambda_{\min}(\bSigma)$. 
	
	\begin{corollarytheorem}\label{corollarytestinglmlzero}
		Under the assumptions of Theorem \ref{theoremtestinglmlzero}, then 
		\begin{align*}
		    \limsup_{n \to \infty} \phzero(\phi(\fhatlzero) \leq \alpha) \leq \alpha
		\end{align*}
		In addition, if $h$ is sufficiently large (not depending on $n$), then
		\begin{align*}
			\liminf_{n \to \infty} \phone(\phi(\fhatlzero) \leq \alpha) > \alpha
		\end{align*}
		 for the hypotheses testing problem in equation \eqref{hypothesesalternativelm} and also for the hypotheses in equation \eqref{hypothesesalternativelmsg} if Assumption \ref{assumptiontestingkoltchinskii} holds.
	\end{corollarytheorem}
	
	\begin{remark}
		If $\mathbf{\Sigma} = \identity_{p}$, then we are interested in testing if $\Vert \betastar \Vert_{2}^{2} = 0$.  Our results should be compared with the minimax lower bound of \cite{guo2019}, who show that the minimax lower bound of the estimation error of $\Vert \betastar \Vert_{2}^{2}$ is $n^{-1/2} + \sstar n^{-1}\log(p)$ over all $\sstar$-sparse vectors with bounded Euclidean norms.  However, under Assumption \ref{assumptiontestingkoltchinskii}, we are able to test at a faster rate since $\betastar = \zerovec_{p}$ is a super-efficient point in the parameter space.  
	\end{remark}
	
	\subsection{Low-Rank Trace Regression}\label{sectiontestingtrm}

	Now we return to the main subject of this paper, the low-rank trace regression model \eqref{modeltrm}. 
	Here, we let $\mathscr{F} \defined \{f : \R^{\done \times \dtwo} \to \R \mid f(\Xmat) = \langle \Xmat, \Thetamat \ranglehs, \rank(\Thetamat) \leq \rstar \}$.  Similarly to the setting of high-dimensional linear models, we require the following mild assumption on the covariates and noise.
	
	\begin{assumption}\label{assumptiontestingtrm}
		The vectorized covariate matrix $\vec(\Xmat)  \in \sg_{d_1d_2}(\sgparamx)$ with mean zero and covariance matrix $\mathbf{\Sigma} \in \RR ^ {d_1d_2 \times d_1d_2}$.  The error $\epsilon \in \sg(\sgparamepsilon)$ and has mean zero and variance $\sigmaepsilonsq$.  Moreover, $\Xmat$ is independent of $\epsilon$.  
	\end{assumption}

		\begin{assumptionstar}\label{assumptiontestingtrmkoltchinskii}
		There exists a $\kol > 0$ such that 
		\begin{align*}
			\Vert \langle \vec(\Xmat), \vvec \rangleeuclid \Vert_{\psi_{2}}^{2} \leq \kol \e \langle \vec(\Xmat), \vvec \rangleeuclid^{2}
		\end{align*}
		for all $\vvec \in \R^{d_1d_2}$.  
	\end{assumptionstar}
	
	The corresponding two pairs of contiguous hypotheses we consider for the low-rank trace regression model are
	\begin{align}\label{hypothesesalternativetrm}
	\begin{aligned}
		H_{0}: \Thetastar = \zeromat_{\done \times \dtwo} \vs
		&H_{1}: \rank(\Thetastar) = \rstar > 0, 
		\\ &\phantom{H_{1}: \Thetastar}
        \vec(\Thetastar)^\T \mathbf{\Sigma} \vec(\Thetastar) = h\Big( \frac{1}{\sqrt{n}} + \frac{\rstar d\log(n)}{n} \Big)
	\end{aligned}
	\end{align}
	and
	\begin{align}\label{hypothesesalternativetrmsg}
	\begin{aligned}
		H_{0}: \Thetastar = \zeromat_{\done \times \dtwo}, 
		\vs
		&H_{1}: \rank(\Thetastar) = \rstar > 0, 
		\\ &\phantom{H_{1}: \Thetastar}
		\vec(\Thetastar)^\T \mathbf{\Sigma} \vec(\Thetastar) = h \frac{\rstar d \log(n)}{n}.
	\end{aligned}
	\end{align}
	For any $\pi \in \Pi$, we define the rank-constrained estimator as
	\begin{align*}
		\fhatlzerotrm(\Xmat_{i} ; (\Xmat_{j}, y_{j})_{j=1}^{n} ; \pi) \defined \langle \Xmat_{i}, \Thetahatlzeropi \ranglehs,
	\end{align*}
	where 
	    \begin{equation}
	        \label{eq:theta_zero_pi}
		    \Thetahatlzeropi \defined \Thetahatlzeropi(r) = \argmin_{\Thetamat \in \R^{\done \times \dtwo}, \rank(\Thetamat) \leq \rhat} \sum_{i=1}^{n} ( y_{\pi(i)} - \langle \Xmat_{i}, \Thetamat \ranglehs)^{2}.
	    \end{equation}
	Next, we show that $\fhatlzerotrm(\cdot)$ satisfies Assumption C without any requirement on $\bSigma$. 
	
	\begin{theorem}\label{theoremtestingtrmlzero}
        Suppose that Assumtion \ref{assumptiontestingtrm} holds and that $\rhat \asymp \rstar$ with $\rhat \geq \rstar$ in \eqref{eq:theta_zero_pi}.  Then for some $\delta > 0$, $\fhatlzero$ satisfies (i) and (ii) of Assumption \ref{assumptiontestingconsistency} with some $\ell_{n} = O( \rstar \log(\done\dtwo) + \log(1 / \delta))$.
	\end{theorem}
	
	We can now establish the asymptotic validity of the low-rank test, again through applying Theorems \ref{theoremtestingnull}, \ref{theoremtestingalternative} and Corollary \ref{corollarysgmean}. 
	
	\begin{corollarytheorem}\label{corollarytestingtrmlzero}
		Under the assumptions of Theorem \ref{theoremtestingtrmlzero}, we have that 
		\begin{align*}
		    \limsup_{n \to \infty} \phzero(\phi(\fhatlzero) \leq \alpha) \leq \alpha.
		\end{align*}
		In addition, if $h$ is sufficiently large (not depending on $n$), then
		\begin{align*}
			\liminf_{n \to \infty} \phone(\phi(\fhatlzero) \leq \alpha) > \alpha
		\end{align*}
		for the hypotheses testing problem in equation \eqref{hypothesesalternativetrm} and also for the hypotheses in equation \eqref{hypothesesalternativetrmsg} if Assumption \ref{assumptiontestingtrmkoltchinskii} holds.
	\end{corollarytheorem}

    \subsection{Robustness of the Rank-Constrained Test} \label{sectiontestingtrmmis}
	
	In the previous sections, we assume that our test statistics have been tuned in an oracular fashion, either through $\lambda$ for regularized estimation or through $r$ for rank-constrained estimation.  However, such oracles are not available in practice.  In this section, we consider the performance of the permutation test with a possibly misspecified rank.  Fix $\rtilde$ and define 
	\begin{align*}
		\Thetamattilde \defined \Thetamattilde(\rtilde) = \argmin_{\Thetamat \in \R^{\done \times \dtwo}, \rank(\Thetamat) \leq \rtilde} \e \langle \Xmat, \Thetastar - \Thetamat \ranglehs^{2}.
	\end{align*}
	In words, $\Thetamattilde$ is the best rank-$\rtilde$ approximation to $\Thetastar$ in terms of prediction risk if $1 \leq \rtilde < \rstar$ and $\Thetamattilde = \Thetastar$ if $\rtilde \geq \rstar$. Then given $h > 0$, consider the hypotheses testing problems
	\begin{align}\label{hypothesesalternativetrmmisspecified}
		H_{0}:  \Thetastar = \zeromat_{\done \times \dtwo}
		\vs
		H_{1}:  \vec(\Thetamattilde)^\T \mathbf{\Sigma} \vec(\Thetamattilde) = h\Big(\frac{1}{\sqrt{n}} + \frac{\rtilde d \log(n)}{ n}\Big)
	\end{align}
	and 
	\begin{align}\label{hypothesesalternativetrmmisspecifiedsg}
		H_{0}:  \Thetastar = \zeromat_{\done \times \dtwo}
		\vs
		H_{1}:  \vec(\Thetamattilde)^\T \mathbf{\Sigma} \vec(\Thetamattilde) = h \frac{\rtilde d \log(d)}{n}.
	\end{align}
	Note that the alternative hypotheses above are stated in terms of $\tilde\bTheta$ and thus vary with respect to $\tilde r$. Intuitively, underestimating the rank refrains one from capturing the complete signal. It is thus hopeless to detect the presence of a nonzero $\bTheta ^ *$ if the signal encapsulated by $\tilde\bTheta$ is too weak.
	
	\begin{theorem}\label{theoremtestingtrmlzeromisspecified}
		Consider model \eqref{modeltrm} and choose $r = \rtilde$ in \eqref{eq:theta_zero_pi}. Under Assumption \ref{assumptiontestingtrm}, we have that
		\begin{align*}
			\limsup_{n \to \infty} \phzero(\phi(\fhatlzero) \leq \alpha) \leq \alpha.
		\end{align*}
		In addition, if $h$ is sufficiently large (not depending on $n$), then
		\begin{align*}
			\liminf_{n \to \infty} \phone(\phi(\fhatlzero) \leq \alpha) > \alpha
		\end{align*}
		 for the hypotheses testing problem in equation \eqref{hypothesesalternativetrmmisspecified} and also for the hypotheses in equation \eqref{hypothesesalternativetrmmisspecifiedsg} if Assumption \ref{assumptiontestingtrmkoltchinskii} holds.
	\end{theorem}

	Theorem \ref{theoremtestingtrmlzeromisspecified} should be compared with Corollary \ref{corollarytestingtrmlzero}.  In particular, by considering $\Thetamattilde$ rather than $\Thetastar$, we can still distinguish between the null and the alternative hypotheses as long as the best rank-$\rtilde$ approximation captures sufficient amount of signal; hence, this allows for the situation where the rank of $\Thetastar$ is misspecified. In particular, by setting $\rtilde = \rstar$, the hypotheses in equations \eqref{hypothesesalternativetrmmisspecified} and \eqref{hypothesesalternativetrmmisspecifiedsg} are equivalent to equations \eqref{hypothesesalternativetrm} and \eqref{hypothesesalternativetrmsg} respectively, and Corollary \ref{corollarytestingtrmlzero} can be viewed as a special case of Theorem \ref{theoremtestingtrmlzeromisspecified}.  As another special case, by setting $\rtilde = 1$, we obtain a tuning-parameter free test that allows for testing the best rank-one approximation of $\Thetastar$.  To the best of our knowledge, this is the first test in the high-dimensional literature that is robust to misspecification of the tuning parameter.
	
	Theorem \ref{theoremtestingtrmlzeromisspecified} seems to imply that the test is more likely to detect the signal if $\tilde r$ is large. However, it should be noted that the required minimal power depends linearly on $\rtilde$ while the signal increases at most by a factor of $\rtilde$.  Thus, the rank-one test, being focused on the leading eigenvalue, may have higher efficiency than a test that is more omnidirectional (for example, see \cite{bickel2006tailor}).
	
	The proof of Theorem \ref{theoremtestingtrmlzeromisspecified} relies on the least-squares structure of the rank-constrained estimator; in particular, the vector of fitted values can be written as $\projV \Yvec$ for some $\Vmat \in \R^{\dtwo \times \rtilde}$.  Thus, the result can immediately be extended to sparse high-dimensional linear models with best subset selection.
    
    By comparison, the choice of the tuning parameter $\lambda$ for the lasso and nuclear norm regularized estimator is inherently challenging.  For estimation, the value of $\lambda$ is usually chosen through cross-validation.  However, to the best of our knowledge, it remains open how to tune these regularization parameters for inference. We investigate a few natural methods to perform cross-validation in Section \ref{subsectionsimulationinference}, which do not lead to satisfactory empirical performance.
    
	\section{Simulations}\label{sectionsimulations}
	
	\subsection{Models and Methods}
	
	In this section, we demonstrate the empirical performance, both in terms of estimation and inference, of the rank-constrained estimator on synthetic data.  We assume the model in equation \eqref{modeltrm}, which is reproduced below 
	\begin{align*}
	    y_{i} = \langle \Xmat_{i}, \Thetamat ^ * \ranglehs + \epsilon_{i}.
	\end{align*}
	In our simulations, we set $n = 200$ and $\done = \dtwo = 20$ and let $r^{\ast} \in \{1, 2, 3, 4\}$.  Regarding the design, we consider two distinct settings, corresponding to two common examples of low-rank trace regression:  (i) compressed sensing and (ii) matrix completion.  In the setting of compressed sensing, we let $\vec(\Xmat_{i})$ have independent and identically distributed standard Gaussian entries.  For matrix completion, we let $\{\Xmat_{1}, \dots, \Xmat_{n}\}$ be a uniform random sample from $\{\be_{j} \be_{k}^\T\}_{j \in [\done], k \in [\dtwo]}$ without replacement, where $\be_{j}$ denotes the $j$th standard basis vector.  
	
	In both scenarios, we generate $\epsilon_{i}$ as independent and identically distributed standard Gaussian random variables.  Finally, we define the signal to noise ratio, denoted by ``SNR,'' as the variance of $\langle \Xmat_{i}, \Thetastar \ranglehs$; the value of SNR is a monotonic function of the power represented by $h$ in equations \eqref{hypothesesalternative} and \eqref{hypothesesalternativesg}.  For in-sample prediction, we consider a logarithmic scale and let $\mathrm{SNR} \in \{1, 1.43, 2.04, 2.92, 4.18, 5.98, 8.55, 12.23, 17.48, 25\}$ and, for inference, we let $\mathrm{SNR} \in \{0, 0.125, 0.25, 0.375, 0.5, 0.75, \allowbreak 1, 2\}$.  To achieve this, we first generate $r^{\ast}$ values uniformly from $(-1, 1)$ to form a diagonal matrix $\mathbf{\Lambda}$.  Then, we draw $\Umat^{\ast}$ and $\Vmat^{\ast}$ from the uniform Haar measure on the Stiefel manifold of dimension $\done \times r^{\ast}$ and $\dtwo \times r^{\ast}$ respectively and set $\Thetastar = \Umat^{\ast} \mathbf{\Lambda} (\Vmat^{\ast})^\T$.  Finally, we scale $\Thetastar$ such that $\vec(\Thetastar)^\T \mathbf{\Sigma} \vec(\Thetastar) = \mathrm{SNR}$.
	
	\subsection{In-Sample Prediction}
	
	For estimation, we compare the in-sample prediction risk of the rank-constrained estimator with that of an oracle least-squares estimator (LS) and the nuclear norm regularized estimator (NN) from equation \eqref{eq:nuclear_ols}.  The oracle least-squares estimator has access to the right singular space $\Vmat^{\ast}$.  Computationally, we use alternating minimization (AM) to approximate the rank-constrained estimator (for example, see \cite{hastie2015matrix} and the references therein).  We employ multiple restarts to avoid local stationary points, using a coarse grid of nuclear norm estimators and a spectral estimator to initialize the AM algorithm; this yields a total of six initializations.  Then, our final rank-constrained estimator is the one that minimizes equation \eqref{eq:rank_ols} out of the six different initializations.  To avoid misspecification of the tuning parameter for both estimators ($r$ for the rank-constrained estimator and $\lambda$ for the nuclear norm estimator), we consider oracle tuning parameters.  To accomplish this, we run both estimators over a grid of tuning parameters for each setting over 1000 Monte Carlo experiments and choose the tuning parameter that yields the minimum prediction risk, which is defined as in \eqref{eq:in_sample_pred}. 
	
	The results of our simulation are presented in Tables \ref{tableesttrmz} and \ref{tableestmc} and Figures \ref{plot_est_z} and \ref{plot_est_mc}.  In general, we see that the performance of the rank-constrained estimator relative to the nuclear norm regularized estimator improves as SNR increases.  This is consistent with the simulation results of \cite{hastie2020best}, who noticed that best subset selection outperforms the lasso for high SNR regimes in high-dimensional linear models.  
	
	\input{RTEstTRMz}

	\begin{figure}[h]
	    \centering
	    \includegraphics[scale=0.145]{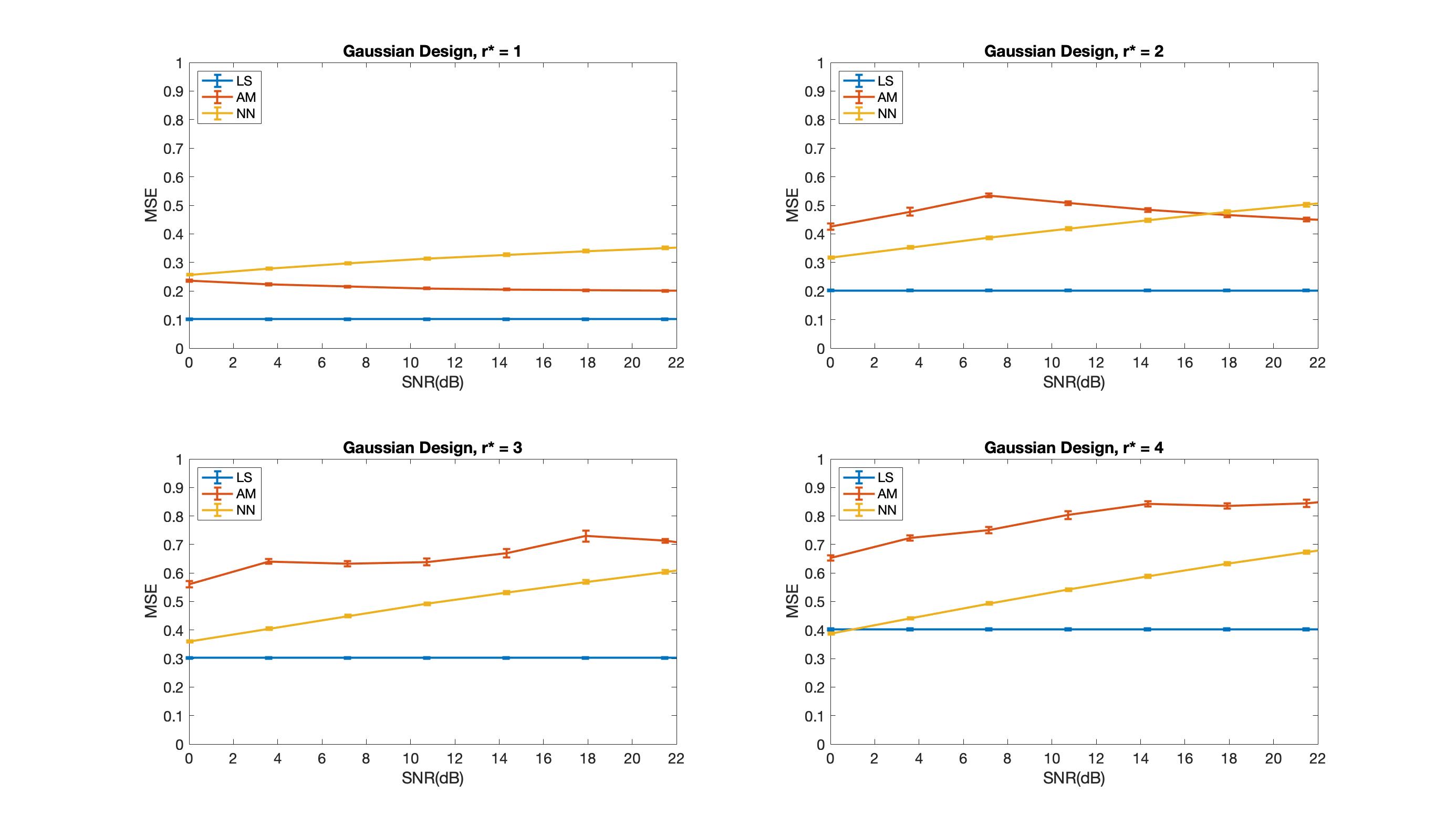}
	    \caption{Plots of in-sample prediction error for Gaussian design}
	    \label{plot_est_z}
	\end{figure}
% 	\begin{figure}
% 	    \centering
% 	    \includegraphics[scale=0.145]{Est_MC.jpg}
% 	    \caption{Plots of in-sample prediction error for matrix completion}
% 	    \label{plot_est_mc}
% 	\end{figure}
	
	\subsection{Inference}\label{subsectionsimulationinference}
	
	For inference, we evaluate the performance of our permutation approach from Section \ref{sectiontesting} and consider the permutation test using both alternating minimization and nuclear norm regularization.  Throughout, we are testing at level $\alpha = 0.05$, and our permutation tests randomly draw nineteen permutations from $\Pi \backslash \{\pizero\}$.  To provide a benchmark for the performance of our testing procedure, we consider two oracles that have access to the right singular space $\Vmat^{\ast}$:  (i) an oracle that uses the low-dimensional $F$-test (FT) and (ii) an oracle that uses our permutation test using least-squares estimation (LS).  
	
	For nuclear norm regularization, we consider four procedures to choose $\lambda$.  First, we consider oracle tuning.  
	Since we use 100 Monte Carlo experiments, when choosing the oracle value of $\lambda$ for nuclear norm regularization, we only consider the values of $\lambda$ for which the number of rejections under the null hypothesis ($\mathrm{SNR} = 0$) is less than or equal to nine.  The nine arises from constructing a confidence interval for $\alpha$ based on 100 independent Bernoulli experiments with success probability $0.05$ as it is two standard errors above $0.05$.  The other three approaches to choosing $\lambda$ are all variations on five-fold cross-validation.  The first procedure, denoted DS for ``data splitting,'' splits the data into two halves, selecting $\lambda$ on the first half by cross-validation and using the selected value of $\lambda$ on the second half to compute $\Lambda^{(\pi)}$ for all $\pi \in \Pi$.  We use half of the data to select $\lambda$ since we need sufficient observations in both halves to estimate rank-three and rank-four matrices.  The second procedure, denoted IS for ``in-sample,'' performs cross-validation for each $\pi \in \Pi$ to obtain $\hat{\lambda}^{(\pi)}$.  After $\hat{\lambda}^{(\pi)}$ is selected, the model is refit using all the observations to obtain in-sample predictions for $\Lambda^{(\pi)}$.  Finally, the third procedure, denoted OS for ``out-of-sample,'' also performs cross-validation for each $\pi \in \Pi$, obtaining five values of $\hat{\lambda}^{(\pi)}$, one for each of the five folds.  Instead of refitting the model as before, we use out-of-sample predicted values when computing $\Lambda^{(\pi)}$.
% 	Instead of refitting the model as before, we use the out-of-sample predicted values for each of the folds.
	
	For alternating minimization, we report all the results for $r \in \{1, 2, 3, 4\}$.  Note that we are using the oracle value of $\lambda$ for nuclear norm regularization.  Thus, we view this as a theoretical benchmark with which to compare the rank-constrained estimator for testing.
	
	The results are presented in Tables \ref{tableinftrmz} and \ref{tableinftrmmc} and Figures \ref{plot_inf_z} and \ref{plot_inf_mc}. We put Table \ref{tableinftrmmc} and Figure \ref{plot_inf_mc} in the supplementary materials to save the space here. We note that, as the SNR increases for a fixed rank, the power of our testing procedure increases.  In general, even under misspecification of the tuning parameter for the rank-constrained estimator, we are able to maintain nominal coverage.  Moreover, even when $\rstar > 1$, it seems that the rank-constrained estimator with $r = 1$ has comparable performance to the optimal nuclear norm regularized estimator as well as permutation testing with larger values of $r$.  Thus, even without any oracular knowledge of $\Thetastar$, we obtain a valid and powerful test that is tuning parameter free by using the rank-constrained estimator with $r = 1$, which is consistent with Theorem \ref{theoremtestingtrmlzeromisspecified}.  
	
	However, when $\lambda$ is chosen via cross-validation, the performance of the nuclear-norm regularized estimator degrades significantly relative to the oracle.  For data-splitting, which has the best empirical performance for non-oracle nuclear norm regularization, we lose half of our observations to selecting $\lambda$ and, compared with the cross-fitting of \cite{chernozhukov2018double}, we cannot switch the roles of the two halves of the dataset.  For the remaining two settings, where $\hat{\lambda}$ depends on $\pi \in \Pi$, we empirically notice that $\hat{\lambda}^{(\pi)} < \hat{\lambda}^{(\pizero)}$ for $\pi \neq \pizero$.  This suggests that $\Lambdapi$ compensates the poorer model fit compared with $\Lambdapizero$ by increasing the complexity of the model, thus enabling overfitting.
	
	\input{RTInfTRMz_CV}

    \begin{figure}[h]
	    \centering
	    \includegraphics[scale=0.2]{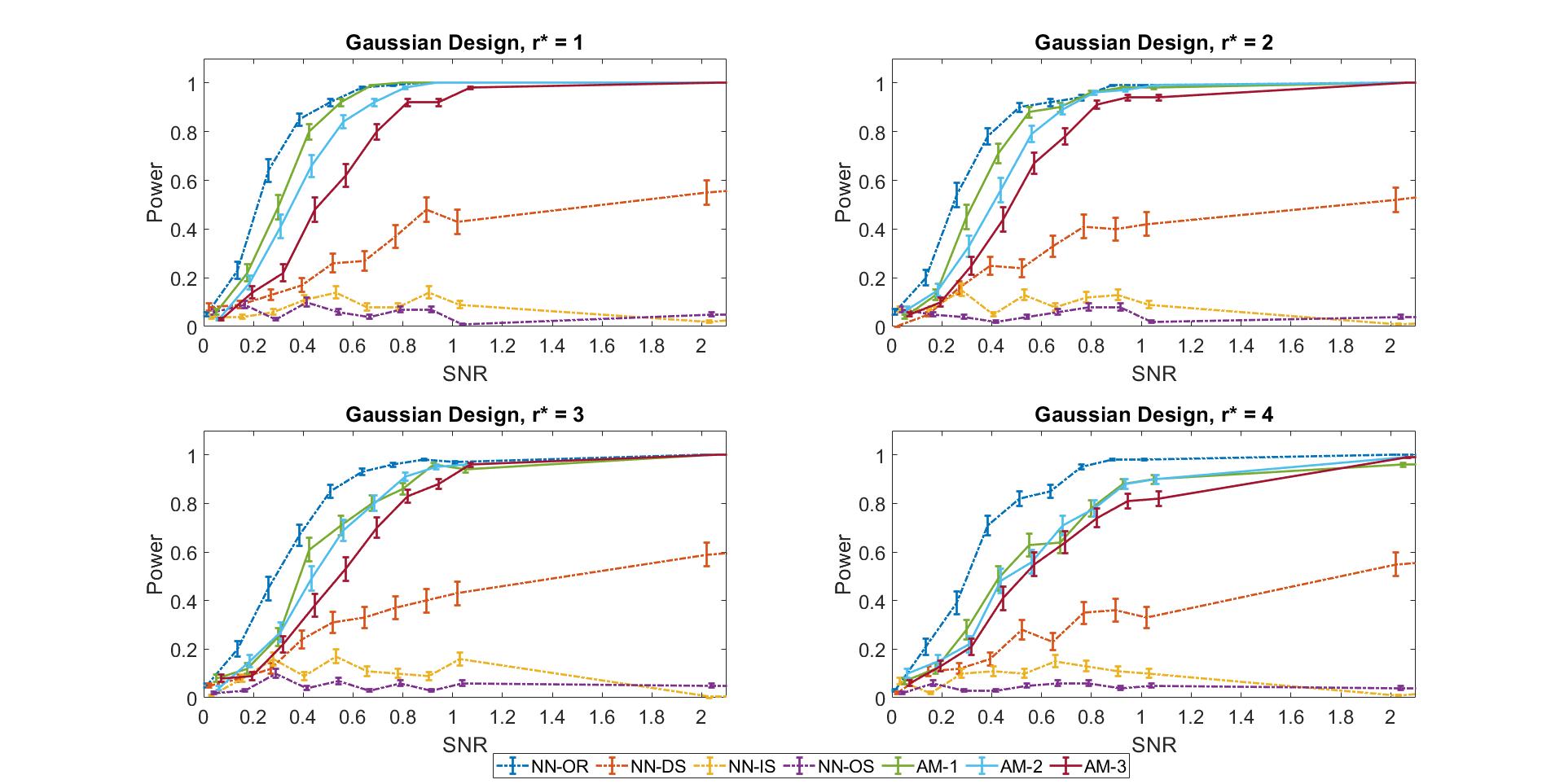}
	    \caption{Plots of power for Gaussian design}
	    \label{plot_inf_z}
	\end{figure}
% 	\begin{figure}
% 	    \centering
% 	    \includegraphics[scale=0.2]{Inf_MC_CV.jpg}
% 	    \caption{Plots of power for matrix completion}
% 	    \label{plot_inf_mc}
% 	\end{figure}
% 	\input{RTInfTRMmc}

%% file: RTEstTRMz.tex
% latex table generated in R 4.1.0 by xtable 1.8-4 package
% Thu Nov 25 10:51:37 2021
\begin{table}[t]
\centering
\caption{Simulations for In-Sample Prediction Risk for Gaussian Design} 
\label{tableesttrmz}
\begin{tabular}{|l|l|rrrrrrrrrr|}
   \hline
 & {SNR} & 1.00 & 1.43 & 2.04 & 2.92 & 4.18 & 5.98 & 8.55 & 12.23 & 17.48 & 25.00 \\ 
   \hline
 & LS & 0.10 & 0.10 & 0.10 & 0.10 & 0.10 & 0.10 & 0.10 & 0.10 & 0.10 & 0.10 \\ 
  $r^\ast = 1$ & AM & 0.24 & 0.22 & 0.22 & 0.21 & 0.21 & 0.20 & 0.20 & 0.20 & 0.20 & 0.20 \\ 
   & NN & 0.26 & 0.28 & 0.30 & 0.31 & 0.33 & 0.34 & 0.35 & 0.36 & 0.37 & 0.38 \\ 
   \hline
 & LS & 0.20 & 0.20 & 0.20 & 0.20 & 0.20 & 0.20 & 0.20 & 0.20 & 0.20 & 0.20 \\ 
  $r^\ast = 2$ & AM & 0.43 & 0.48 & 0.53 & 0.51 & 0.48 & 0.47 & 0.45 & 0.44 & 0.43 & 0.42 \\ 
   & NN & 0.32 & 0.35 & 0.39 & 0.42 & 0.45 & 0.48 & 0.50 & 0.53 & 0.55 & 0.57 \\ 
   \hline
 & LS & 0.30 & 0.30 & 0.30 & 0.30 & 0.30 & 0.30 & 0.30 & 0.30 & 0.30 & 0.30 \\ 
  $r^\ast = 3$ & AM & 0.56 & 0.64 & 0.63 & 0.64 & 0.67 & 0.73 & 0.71 & 0.69 & 0.67 & 0.65 \\ 
   & NN & 0.36 & 0.41 & 0.45 & 0.49 & 0.53 & 0.57 & 0.60 & 0.64 & 0.67 & 0.69 \\ 
   \hline
 & LS & 0.40 & 0.40 & 0.40 & 0.40 & 0.40 & 0.40 & 0.40 & 0.40 & 0.40 & 0.40 \\ 
  $r^\ast = 4$ & AM & 0.65 & 0.72 & 0.75 & 0.80 & 0.84 & 0.84 & 0.84 & 0.87 & 0.90 & 0.88 \\ 
   & NN & 0.39 & 0.44 & 0.49 & 0.54 & 0.59 & 0.63 & 0.67 & 0.71 & 0.75 & 0.78 \\ 
   \hline
\end{tabular}
\end{table}

%% file: RTInfTRMz_CV.tex
% latex table generated in R 4.1.0 by xtable 1.8-4 package
% Thu Mar 24 10:51:06 2022
\begin{table}[H]
\centering
\caption{Simulations for Inference for Gaussian Design} 
\label{tableinftrmz}
\begin{tabular}{|l|l|rrrrrrrrrr|}
   \hline
 & SNR & 0.000 & 0.125 & 0.250 & 0.375 & 0.500 & 0.625 & 0.750 & 0.875 & 1.000 & 2.000\\ 
   \hline
 & FT & 0.02 & 0.81 & 1.00 & 1.00 & 1.00 & 1.00 & 1.00 & 1.00 & 1.00 & 1.00 \\ 
   & LS & 0.06 & 0.74 & 0.99 & 1.00 & 1.00 & 1.00 & 1.00 & 1.00 & 1.00 & 1.00 \\ 
   & NN-OR & 0.05 & 0.23 & 0.64 & 0.85 & 0.92 & 0.98 & 0.99 & 1.00 & 1.00 & 1.00 \\ 
   & NN-DS & 0.08 & 0.09 & 0.13 & 0.17 & 0.26 & 0.27 & 0.37 & 0.48 & 0.43 & 0.55\\ 
   & NN-IS & 0.04 & 0.04 & 0.06 & 0.11 & 0.14 & 0.08 & 0.08 & 0.14 & 0.09 & 0.02 \\ 
  $r^\ast = 1$ & NN-OS & 0.06 & 0.09 & 0.03 & 0.10 & 0.06 & 0.04 & 0.07 & 0.07 & 0.01 & 0.05 \\ 
   & AM-1 & 0.06 & 0.22 & 0.49 & 0.80 & 0.92 & 0.99 & 1.00 & 1.00 & 1.00 & 1.00 \\ 
   & AM-2 & 0.03 & 0.18 & 0.41 & 0.66 & 0.84 & 0.92 & 0.98 & 1.00 & 1.00 & 1.00\\ 
   & AM-3 & 0.03 & 0.14 & 0.22 & 0.48 & 0.62 & 0.80 & 0.92 & 0.92 & 0.98 & 1.00 \\ 
   & AM-4 & 0.07 & 0.10 & 0.15 & 0.28 & 0.45 & 0.60 & 0.72 & 0.75 & 0.83 & 1.00 \\ 
%   & AM-5 & 0.04 & 0.06 & 0.11 & 0.21 & 0.24 & 0.28 & 0.41 & 0.47 & 0.56 & 0.85 & 1.00 \\ 
   \hline
 & FT & 0.02 & 0.65 & 0.99 & 1.00 & 1.00 & 1.00 & 1.00 & 1.00 & 1.00 & 1.00 \\ 
   & LS & 0.01 & 0.54 & 0.92 & 1.00 & 1.00 & 1.00 & 1.00 & 1.00 & 1.00 & 1.00 \\ 
   & NN-OR & 0.06 & 0.20 & 0.54 & 0.78 & 0.90 & 0.92 & 0.94 & 0.99 & 0.99 & 1.00 \\ 
   & NN-DS & 0.00 & 0.05 & 0.16 & 0.25 & 0.24 & 0.33 & 0.41 & 0.40 & 0.42 & 0.52  \\ 
   & NN-IS & 0.07 & 0.07 & 0.15 & 0.05 & 0.13 & 0.08 & 0.12 & 0.13 & 0.09 & 0.01  \\ 
  $r^\ast = 2$ & NN-OS & 0.07 & 0.05 & 0.04 & 0.02 & 0.04 & 0.06 & 0.08 & 0.08 & 0.02 & 0.04  \\ 
   & AM-1 & 0.04 & 0.13 & 0.45 & 0.71 & 0.88 & 0.90 & 0.96 & 0.98 & 0.98 & 1.00  \\ 
   & AM-2 & 0.07 & 0.15 & 0.33 & 0.56 & 0.79 & 0.89 & 0.96 & 0.97 & 0.99 & 1.00  \\ 
   & AM-3 & 0.05 & 0.10 & 0.25 & 0.44 & 0.67 & 0.78 & 0.91 & 0.94 & 0.94 & 1.00  \\ 
   & AM-4 & 0.07 & 0.10 & 0.18 & 0.25 & 0.41 & 0.54 & 0.64 & 0.78 & 0.77 & 1.00  \\ 
%   & AM-5 & 0.05 & 0.08 & 0.15 & 0.17 & 0.26 & 0.26 & 0.30 & 0.33 & 0.47 & 0.76 & 1.00 \\ 
   \hline
 & FT & 0.02 & 0.56 & 0.89 & 1.00 & 1.00 & 1.00 & 1.00 & 1.00 & 1.00 & 1.00  \\ 
   & LS & 0.03 & 0.48 & 0.84 & 0.96 & 1.00 & 1.00 & 1.00 & 1.00 & 1.00 & 1.00  \\ 
   & NN-OR & 0.05 & 0.20 & 0.45 & 0.67 & 0.85 & 0.93 & 0.96 & 0.98 & 0.97 & 1.00  \\ 
   & NN-DS & 0.05 & 0.08 & 0.12 & 0.24 & 0.31 & 0.33 & 0.37 & 0.40 & 0.43 & 0.59  \\ 
   & NN-IS & 0.01 & 0.08 & 0.16 & 0.09 & 0.17 & 0.11 & 0.10 & 0.09 & 0.16 & 0.00  \\ 
  $r^\ast = 3$ & NN-OS & 0.02 & 0.03 & 0.10 & 0.04 & 0.07 & 0.03 & 0.06 & 0.03 & 0.06 & 0.05  \\ 
   & AM-1 & 0.08 & 0.12 & 0.25 & 0.61 & 0.71 & 0.80 & 0.86 & 0.96 & 0.94 & 1.00  \\ 
   & AM-2 & 0.04 & 0.15 & 0.27 & 0.49 & 0.69 & 0.80 & 0.91 & 0.95 & 0.96 & 1.00  \\ 
   & AM-3 & 0.08 & 0.09 & 0.22 & 0.38 & 0.53 & 0.70 & 0.83 & 0.88 & 0.96 & 1.00  \\ 
   & AM-4 & 0.08 & 0.09 & 0.15 & 0.27 & 0.44 & 0.39 & 0.57 & 0.68 & 0.80 & 1.00  \\ 
%   & AM-5 & 0.05 & 0.11 & 0.05 & 0.16 & 0.24 & 0.27 & 0.31 & 0.35 & 0.32 & 0.79 & 1.00 \\ 
   \hline
 & FT & 0.02 & 0.43 & 0.84 & 0.98 & 1.00 & 1.00 & 1.00 & 1.00 & 1.00 & 1.00  \\ 
   & LS & 0.05 & 0.40 & 0.75 & 0.93 & 0.99 & 1.00 & 1.00 & 1.00 & 1.00 & 1.00  \\ 
   & NN-OR & 0.03 & 0.21 & 0.39 & 0.71 & 0.82 & 0.85 & 0.95 & 0.98 & 0.98 & 1.00  \\ 
   & NN-DS & 0.02 & 0.11 & 0.12 & 0.16 & 0.28 & 0.23 & 0.35 & 0.36 & 0.33 & 0.55  \\ 
   & NN-IS & 0.07 & 0.02 & 0.10 & 0.11 & 0.10 & 0.15 & 0.13 & 0.11 & 0.10 & 0.01  \\ 
  $r^\ast = 4$ & NN-OS & 0.02 & 0.06 & 0.03 & 0.03 & 0.05 & 0.06 & 0.06 & 0.04 & 0.05 & 0.04  \\ 
   & AM-1 & 0.07 & 0.12 & 0.28 & 0.49 & 0.63 & 0.64 & 0.78 & 0.88 & 0.90 & 0.96  \\ 
   & AM-2 & 0.10 & 0.15 & 0.22 & 0.48 & 0.56 & 0.71 & 0.78 & 0.88 & 0.90 & 0.99  \\ 
   & AM-3 & 0.06 & 0.13 & 0.21 & 0.41 & 0.55 & 0.64 & 0.74 & 0.81 & 0.82 & 0.99  \\ 
   & AM-4 & 0.05 & 0.08 & 0.20 & 0.30 & 0.33 & 0.52 & 0.53 & 0.61 & 0.73 & 0.97  \\ 
%   & AM-5 & 0.05 & 0.10 & 0.06 & 0.18 & 0.25 & 0.30 & 0.33 & 0.28 & 0.36 & 0.67 & 0.94 \\ 
   \hline
\end{tabular}
\end{table}

%% file: appendix.tex
	\section{Proofs}\label{sectionproofs}
	
	\begin{proof}[Proof of Lemma \ref{lem1}]
	By definition of $\Thetahatlzero$, we have that 
		\begin{align*}
			\frac{1}{n} \sum_{i=1}^{n} ( y_{i} - \langle \Xmat_{i}, \Thetahatlzero \ranglehs )^{2} \leq \frac{1}{n} \sum_{i=1}^{n} (y_{i} - \langle \Xmat_{i}, \Thetastar \ranglehs )^{2},
		\end{align*}
		which implies that 
		\begin{align*}
			\frac{1}{n} \sum_{i=1}^{n} \langle \Xmat_{i}, \Thetahatlzero - \Thetastar \ranglehs^{2} \leq \frac{2}{n} \sum_{i=1}^{n} \epsilon_{i} \langle \Xmat_{i}, \Thetahatlzero - \Thetastar \ranglehs.
		\end{align*}
		If $n^{-1} \sum_{i=1}^{n} \langle \Xmat_{i}, \Thetahatlzero - \Thetastar \ranglehs^{2} = 0$, then the result follows.  Therefore, we only consider the case where $n^{-1} \sum_{i=1}^{n} \langle \Xmat_{i}, \Thetahatlzero - \Thetastar \ranglehs^{2} > 0$.  Dividing both sides of the above display by $(n^{-1} \sum_{i=1}^{n} \langle \Xmat_{i}, \Thetahatlzero - \Thetastar \ranglehs^{2})^{1/2}$ yields
		\begin{align*}
			\Big(\frac{1}{n} \sum_{i=1}^{n} \langle \Xmat_{i}, \Thetahatlzero - \Thetastar \ranglehs^{2} \Big)^{1/2} 
			&\leq \left(\frac{4}{n}\right)^{1/2} \frac{\sum_{i=1}^{n} \epsilon_{i} \langle \Xmat_{i}, \Thetahatlzero - \Thetastar \ranglehs}{(\sum_{i=1}^{n} \langle \Xmat_{i}, \Thetahatlzero - \Thetastar \ranglehs^{2})^{1/2}} \\ 
			&\leq \left(\frac{4}{n}\right)^{1/2} \sup_{\substack{\Mmat \in \R^{\done \times \dtwo} \\ \rank(\Mmat) \leq 2r \\ \sum_{i=1}^{n} \langle \Xmat_{i}, \Mmat \ranglehs^{2} > 0}} \frac{\sum_{i=1}^{n} \epsilon_{i} \langle \Xmat_{i}, \Mmat \ranglehs}{(\sum_{i=1}^{n} \langle \Xmat_{i}, \Mmat \ranglehs^{2})^{1/2}}.
		\end{align*}
		The second inequality follows from the fact that $\rank(\Thetahatlzero - \Thetastar) \leq r + \rstar \leq 2r$.  Now, for any $\Mmat$ satisfying the above, there exist matrices $\Umat \in \R^{\done \times 2r}$ and $\Vmat \in \R^{\dtwo \times 2r}$ such that $\Mmat = \Umat\Vmat^\T$.  Note that $\langle \Xmat_{i}, \Umat\Vmat^\T \ranglehs = \langle \Xmat_{i}\Vmat, \Umat \ranglehs$.  Let $\Xmat_{\Vmat} \in \R^{n \times 2r\done}$ be the matrix whose $i$th row is $\vec(\Xmat_{i}\Vmat)$ and $\gammaU \defined \vec(\Umat) \in \R^{2r\done}$.  Denote by $\projV \in \R^{n \times n}$ the projection operator onto the column space of $\Xmat_{\Vmat}$.  Therefore, we may further bound the above display by
		\begin{align}\label{equationcs}
    		\begin{aligned}
    		    \Big( \frac{1}{n} \sum_{i=1}^{n} \langle \Xmat_{i}, \Thetahatlzero - \Thetastar \ranglehs^{2} \Big)^{1/2} 
    			&\leq \left(\frac{4}{n}\right)^{1/2} \sup_{\substack{\Umat \in \R^{\done \times 2r} \\ \Vmat \in \R^{\dtwo \times 2r} \\ \sum_{i=1}^{n} \langle \Xmat_{i}\Vmat, \Umat \ranglehs^{2} > 0}} \frac{\sum_{i=1}^{n} \epsilon_{i} \langle \Xmat_{i}\Vmat, \Umat \ranglehs}{(\sum_{i=1}^{n} \langle \Xmat_{i}\Vmat, \Umat \ranglehs^{2})^{1/2}} \\ 
    			&\leq \left(\frac{4}{n}\right)^{1/2} \sup_{\substack{\Umat \in \R^{\done \times 2r} \\ \Vmat \in \R^{\dtwo \times 2r} \\ \Vert \Xmat_{\Vmat} \gammaU \Vert_{2} > 0}} \frac{\epsilonvec^\T \Xmat_{\Vmat} \gammaU}{\Vert \Xmat_{\Vmat} \gammaU \Vert_{2}} \\ 
    			&\leq \left(\frac{4}{n}\right)^{1/2} \sup_{\Vmat \in \R^{\dtwo \times 2r}} \Vert \projV \epsilonvec \Vert_{2},
    		\end{aligned}
		\end{align}
		where the last line follows from the Cauchy-Schwarz inequality and the identity $\projV \Xmat_{\Vmat} = \Xmat_{\Vmat}$. The conclusion follows immediately by squaring both sides.
	\end{proof}

    \def\aone{a_{1}}
    \def\atwo{a_{2}}
    \def\athree{a_{3}}
    \def\afour{a_{4}}
	\let\cone\aone 
	\let\ctwo\atwo
	\let\cthree\athree
	\let\cfour\afour
	\begin{proof}[Proof of Theorem \ref{theorempredictionrisk}]
		Now, for a fixed $\Vmat \in \R^{\dtwo \times 2r}$, there exists a matrix $\Vmattilde \in \netepsilon(\projset)$ such that $\Vert \projV - \projVtilde \Verths \leq \varepsilon$.  Then, 
		\begin{align*}
			\Vert \projV \epsilonvec \Vert_{2}^{2} 
			\leq 2 \Vert (\projV - \projVtilde) \epsilonvec \Vert_{2}^{2} + 2 \Vert \projVtilde \epsilonvec \Vert_{2}^{2} 
			\leq 2 \varepsilon^{2} \Vert \epsilonvec \Vert_{2}^{2} + 2 \Vert \projVtilde \epsilonvec \Vert_{2}^{2}.
		\end{align*}
		Define $\Tset = \Tset_{n}$ as 
		\begin{align*}
			\Tset \defined \bigcap_{\Vmattilde \in \netepsilon(\projset)} \{\Vert \projVtilde \epsilonvec \Vert_{2}^{2} \leq \ctwo r \max(\done, \dtwo \log(\done n^{3} / \varepsilon)) + 2\sigmaepsilonsq r\dtwo \} \cap \{ \Vert \epsilonvec \Vert_{2}^{2} \leq \cone n + n\sigmaepsilonsq \}
		\end{align*}
		for some constants $\cone, \ctwo \geq \max(1, \sgparamepsilon^{2})$ to be chosen later.  By the Hanson-Wright inequality (Theorem 1.1 of \cite{rudelson2013}), it follows that
		\begin{align*}
			&\p(\Vert \epsilonvec \Vert_{2}^{2} > t + \sigmaepsilonsq n) \leq 2 \exp \left[-\cthree \min \left( \frac{t^{2}}{n \sgparamepsilon^{4}}, \frac{t}{\sgparamepsilon^{2}} \right) \right], 
			\\ 
			&\p(\Vert \projVtilde \epsilonvec \Vert_{2}^{2} > t + 2\sigmaepsilonsq r\dtwo) \leq 2 \exp \left[-\cthree \min \left( \frac{t^{2}}{2r \dtwo \sgparamepsilon^{4}}, \frac{t}{\sgparamepsilon^{2}} \right) \right],
		\end{align*}
		for some universal constant $\cthree > 0$.  Hence, a union bound implies
		\begin{align*}
			\p(\Tset^\C) &\leq 2 \exp \left[ -\ctwo \cthree \sgparamepsilon^{-2} r \max (\done, \dtwo \log(\done n^{3} / \varepsilon)) + \covnumepsilon(\projset) \right] + 2\exp [ -\cone \cthree \sgparamepsilon^{-2} n ] \\ 
			&\leq 2 \exp \left[ -\ctwo \cthree \sgparamepsilon^{-2} r \max (\done, \dtwo \log(\done n^{3} / \delta)) + 2r\done \log(2) + (2r\dtwo + 1) \log(24r\done n^{3} / \varepsilon) \right] \\ 
			&\phantom{\leq} + 2\exp [ -\cone \cthree \sgparamepsilon^{-2} n ] \\ 
			&\leq 2 \exp[-\cfour r \max(\done, \dtwo \log(\done n^{3} / \varepsilon))] + 2 \exp[-\cfour n].
		\end{align*}
		for some constant $\cfour > 0$ depending on $\sgparamepsilon$, $\cone$, $\ctwo$, and $\cthree$.  Now, on the event $\Tset$, it follows that 
		\begin{align*}
			\Vert \projV \epsilonvec \Vert_{2}^{2} 
			\leq 2 \varepsilon^{2} (\cone + \sigmaepsilonsq) n + 2 \ctwo r \max(\done, \dtwo \log(\done n^{3} / \varepsilon)) + 4 r \dtwo \sigmaepsilonsq.
		\end{align*}
		Letting $\varepsilon = rn^{-1} \max(\done, \dtwo)$, we have 
		\begin{align*}
			\Vert \projV \epsilonvec \Vert_{2}^{2} 
			\leq 2 (\cone + \sigmaepsilonsq) r^{2} n^{-1} \max(\done^{2}, \dtwo^{2}) + 2 \ctwo r \max(\done, \dtwo \log(r^{-1}n^{2}\min(1, \done \dtwo^{-1}))) + 4 r \dtwo \sigmaepsilonsq.
		\end{align*}
		Since this holds for an arbitrary $\Vmat \in \R^{\dtwo \times 2r}$, we conclude that 
		\begin{align*}
			\frac{1}{n} \sum_{i=1}^{n} \langle \Xmat_{i}, \Thetahatlzero - \Thetastar \ranglehs^{2}
			&\leq 8 (\cone + \sigmaepsilonsq) r^{2}\max(\done^{2}, \dtwo^{2})n^{-2} + 16 \sigmaepsilonsq r \dtwo n^{-1} \\ 
			&\phantom{\leq} + 8 \ctwo r \max(\done, \dtwo \log(r^{-1}n^{2}\min(1, \done \dtwo^{-1}))) n^{-1}.
		\end{align*}
		Let $c_{1} = 8 \cone + 24 \sigmaepsilonsq + 8 \ctwo$ and $c_{2} = \cfour$.  Using the fact that $\done \leq \dtwo = d$ and $rd < n$ finishes the proof.
	\end{proof}
	
	\begin{proof}[Proof of Corollary \ref{corollarypredictionrisk}]
	    Let $\Thetatilde$ be defined as 
		\begin{align*}
			\Thetatilde \defined \Thetatilde(\rhat) = \argmin_{\Thetamat \in \R^{\done \times \dtwo}, \rank(\Thetamat) \leq \rhat} \sum_{i=1}^{n} (f(x_{i}) - \langle \Xmat_{i}, \Thetamat \ranglehs)^{2}.
		\end{align*}
		Then, by the definition of $\Thetahatlzero$, we have that 
		\begin{align*}
			\sum_{i=1}^{n} (y_{i} - \langle \Xmat_{i}, \Thetahatlzero \ranglehs)^{2} \leq \sum_{i=1}^{n} (y_{i} - \langle \Xmat_{i}, \Thetatilde \ranglehs)^{2}.
		\end{align*}
		Expanding the square and rearranging yields
		\begin{align*}
			\sum_{i=1}^{n} (f(x_{i}) - \langle \Xmat_{i}, \Thetahatlzero \ranglehs)^{2} \leq \sum_{i=1}^{n} (f(x_{i}) - \langle \Xmat_{i}, \Thetatilde \ranglehs)^{2} + 2 \sum_{i=1}^{n} \epsilon_{i} \langle \Xmat_{i}, \Thetahatlzero - \Thetatilde \ranglehs.
		\end{align*}
		If $\sum_{i=1}^{n} (f(x_{i}) - \langle \Xmat_{i}, \Thetahatlzero \ranglehs)^{2} = 0$, then the result immediately follows.  Hence, for the remainder of the proof, we assume that $\sum_{i=1}^{n} (f(x_{i}) - \langle \Xmat_{i}, \Thetahatlzero \ranglehs)^{2} > 0$.  Now, dividing both sides, it follows that 
		\begin{align*}
			&\Big\{ \sum_{i=1}^{n} (f(x_{i}) - \langle \Xmat_{i}, \Thetahatlzero \ranglehs)^{2} \Big\}^{1/2} 
			\\&\phantom{\leq}\leq
			\frac{\sum_{i=1}^{n} (f(x_{i}) - \langle \Xmat_{i}, \Thetatilde \ranglehs)^{2}}{\Big\{ \sum_{i=1}^{n} (f(x_{i}) - \langle \Xmat_{i}, \Thetahatlzero \ranglehs)^{2} \Big\}^{1/2}} 
			+ 2 \frac{\sum_{i=1}^{n} \epsilon_{i} \langle \Xmat_{i}, \Thetahatlzero - \Thetatilde \ranglehs}{\Big\{ \sum_{i=1}^{n} (f(x_{i}) - \langle \Xmat_{i}, \Thetahatlzero \ranglehs)^{2} \Big\}^{1/2}}.
		\end{align*}
		By the construction of $\Thetatilde$, we deduce that
		\begin{align*}
			\sum_{i=1}^{n} (f(x_{i}) - \langle \Xmat_{i}, \Thetatilde \ranglehs)^{2} \leq \sum_{i=1}^{n} (f(x_{i}) - \langle \Xmat_{i}, \Thetahatlzero \ranglehs)^{2}
		\end{align*}
		and 
		\begin{align*}
			\sum_{i=1}^{n} \langle \Xmat_{i}, \Thetahatlzero - \Thetatilde \ranglehs^{2} 
			&\leq 
			2 \sum_{i=1}^{n} (f(x_{i}) - \langle \Xmat_{i}, \Thetahatlzero \ranglehs)^{2} 
			+ 2 \sum_{i=1}^{n} (f(x_{i}) - \langle \Xmat_{i}, \Thetatilde \ranglehs)^{2} \\ 
			&\leq 4 \sum_{i=1}^{n} (f(x_{i}) - \langle \Xmat_{i}, \Thetahatlzero \ranglehs)^{2}.
		\end{align*}
		Therefore, we have that 
		\begin{align*}
			\frac{\sum_{i=1}^{n} (f(x_{i}) - \langle \Xmat_{i}, \Thetatilde \ranglehs)^{2}}{\Big\{ \sum_{i=1}^{n} (f(x_{i}) - \langle \Xmat_{i}, \Thetahatlzero \ranglehs)^{2} \Big\}^{1/2}} 
			\leq 
			\Big\{ \sum_{i=1}^{n} (f(x_{i}) - \langle \Xmat_{i}, \Thetatilde \ranglehs)^{2} \Big\}^{1/2}.
		\end{align*}
		Moreover, note that $\rank(\Thetahatlzero - \Thetatilde) \leq 2\rhat$; hence, by an analogous argument to equation \eqref{equationcs}, 
% 		by the Cauchy-Schwarz inequality, 
		\begin{align*}
			\sum_{i=1}^{n} \epsilon_{i} \langle \Xmat_{i}, \Thetahatlzero - \Thetatilde \ranglehs 
			&\leq \sup_{\Vmat \in \R^{\dtwo \times 2\rhat}} \Vert \projV \epsilonvec \Vert_{2} \Big\{ \sum_{i=1}^{n} \langle \Xmat_{i}, \Thetahatlzero - \Thetatilde \ranglehs^{2} \Big\}^{1/2} \\ 
			&\leq 2 \sup_{\Vmat \in \R^{\dtwo \times 2\rhat}} \Vert \projV \epsilonvec  \Vert_{2} \Big\{ \sum_{i=1}^{n} (f(x_{i}) - \langle \Xmat_{i}, \Thetahatlzero \ranglehs)^{2} \Big\}^{1/2}.
		\end{align*}
		Combining these calculations, it follows that 
		\begin{align*}
			\Big\{ \sum_{i=1}^{n} (f(x_{i}) - \langle \Xmat_{i}, \Thetahatlzero \ranglehs)^{2} \Big\}^{1/2} 
			\leq 
			\Big\{ \sum_{i=1}^{n} (f(x_{i}) - \langle \Xmat_{i}, \Thetatilde \ranglehs)^{2} \Big\}^{1/2}
			+ 4 \sup_{\Vmat \in \R^{\dtwo \times 2\rhat}} \Vert \projV \epsilonvec \Vert_{2}.
		\end{align*}
		It is left to bound $\sup_{\Vmat \in \R^{\dtwo \times 2\rhat}} \Vert \projV \epsilonvec \Vert_{2}^{2}$, which is provided in the proof of Theorem \ref{theorempredictionrisk}.
	\end{proof}

	\def\cone{c_{1}}
	\def\ctwo{c_{2}}
	\def\cthree{c_{3}}
	\def\cfour{c_{4}}
	\begin{proof}[Proof of Lemma \ref{lemmapermutationmeanremainder}]
		Temporarily fix $\pi \in \Pitilde$ and note that $\fpi(x_{i}) = 0$ for $i \in \Asetpione \cup \Asetpitwo$.  Applying Chebyshev's inequality yields that 
		\begin{align*}
			\sum_{i=1}^{n} [\fpi(x_{i})]^{2} = \sum_{i\in \Asetpithree} [\fpi(x_{i})]^{2} \leq \sum_{i\in \Asetpithree} f^{2}(x_{\pi(i)}) \leq \log^{2}(n) \sigmamusq + \cone \log(n)
		\end{align*}
		with probability at least $1 - \delta$
		for some $\cone > 0$ depending on $\delta$ and $\E f^4(x)$.  We use the fact that $|\Asetpithree| \leq \log^{2}(n)$ from the construction of $\Pitilde$.  Since $\pi \in \Pitilde$ is arbitrary, this finishes the proof.
	\end{proof}

	\begin{proof}[Proof of Theorem \ref{theoremtestingnull}]
		The proof is standard.  For example, see Section 15.2 of \cite{lehmann2006}.
	\end{proof}

	\begin{proof}[Proof of Theorem \ref{theoremtestingalternative}]
		Fix $0 < \delta < \alpha (1 - \alpha) / 4$.  Then, by the triangle inequality, it follows that 
		\begin{align*}
			\Lambdapizero = \sum_{i=1}^{n} [\fhatpizero(x_{i})]^{2} \geq 2^{-1} \sum_{i=1}^{n} [\fpizero(x_{i})]^{2} - \sum_{i=1}^{n} [\fhatpizero(x_{i}) - \fpizero(x_{i})]^{2}.
		\end{align*}
		From Assumption \ref{assumptiontestingmoments}, the Chebyshev's inequality implies that there exists a constant $\tone > 0$ (not depending on $n$) such that, for $n$ sufficiently large, 
		\begin{align}\label{equationmeanberryesseen}
			\sum_{i=1}^{n} [\fpizero(x_{i})]^{2} \geq n\sigmamusq - \tone n^{1/2}
		\end{align}
		with probability at least $1 - \delta$.  Assumption \ref{assumptiontestingconsistency} ensures that 
		\begin{align*}
			\sum_{i=1}^{n} [\fhatpizero(x_{i}) - \fpizero(x_{i})]^{2} \leq \ell_{n}
		\end{align*}
		with probability at least $1 - \delta$.  Hence, it holds with probability at least $1 - 2\delta$ that 
		\begin{align*}
			\Lambdapizero \geq 2^{-1} (n\sigmamusq - \tone n^{1/2}) - \ell_{n}.
		\end{align*}
		Now, temporarily fix $\pi \in \Pitilde$.  Again, by the triangle inequality, it follows that 
		\begin{align*}
			\Lambdapi \leq 2\sum_{i=1}^{n} [\fhatpi(x_{i})]^{2} - \fpi(x_{i})]^{2} + 2 \sum_{i=1}^{n} [\fpi(x_{i})]^{2}
		\end{align*}
		Assumption \ref{assumptiontestingconsistency} implies that 
		\begin{align*}
			\sum_{i=1}^{n} [\fhatpi(x_{i})]^{2} - \fpi(x_{i})]^{2} \leq \ell_{n}
		\end{align*}
		with probability at least $1 - \delta$.  Moreover, we have from Lemma \ref{lemmapermutationmeanremainder} that 
		\begin{align*}
			\sum_{i \in \Asetpithree} [\fpi(x_{i})]^{2} \leq \log^{2}(n) \sigmamusq + \ttwo \log(n)
		\end{align*}
		with probability at least $1 - \delta$ for some constant $\ttwo > 0$.  Hence, with probability at least $1 - 2\delta$, 
		\begin{align*}
			\Lambdapi \leq 2\ell_{n} + 2 \log^{2}(n) \sigmamusq + 2\ttwo \log(n).
		\end{align*}
		Combining the above calculations, for $n$ sufficiently large, 
		\begin{align*}
			\Lambdapizero - \Lambdapi 
			&\geq 2^{-1} (n\sigmamusq - \tone n^{1/2}) - 3\ell_{n} - 2 \log^{2}(n) \sigmamusq - 2\ttwo \log(n) \\ 
			&\geq 2^{-1} \{h(n^{1/2} + \ell_{n}) - \tone n^{1/2} \} - 3\ell_{n} - 2h (n^{-1/2} + \ell_{n} n^{-1}) \log^{2}(n) - 2\ttwo \log(n) \\ 
			&> 0
		\end{align*}
		with probability at least $1 - 4\delta$ if $h > 0$ is sufficiently large (not depending on $n$).  Thus, 
		\begin{align*}
			\phone(\Lambdapizero > \Lambdapi) \geq 1 - 4\delta
		\end{align*}
		for $n$ sufficiently large.  Since $\pi$ is arbitrary, it follows that 
		\begin{align*}
			\liminf_{n \to \infty} \min_{\pi \in \Pitilde} \phone(\Lambdapizero > \Lambdapi) \geq 1 - 4\delta.
		\end{align*}
		Hence, 		
		\begin{align*}
			\limsup_{n \to \infty} \e_{H_{1}} \phi 
			&= \limsup_{n \to \infty} |\Pi|^{-1} \e_{H_{1}} \sum_{\pi \in \Pi} \indic{\Lambdapizero \leq \Lambdapi} \\ 
			&= 1 - \liminf_{n \to \infty}  |\Pi|^{-1} \sum_{\pi \in \Pitilde} \phone(\Lambdapizero > \Lambdapi) \\ 
			&\leq 4\delta.
		\end{align*}
		Since $\delta < \alpha (1-\alpha) / 4$, this finishes the proof.
	\end{proof}
	
	\begin{proof}[Proof of Corollary \ref{corollarysgmean}]
		By Chebyshev's inequality, there exists a constant $\tthree > 0$ such that 
		\begin{align*}
			\sum_{i=1}^{n} [\fpizero(x_{i})]^{2} \geq n\sigmamusq - \tthree n^{1/2} \sigmamusq
		\end{align*}
		with probability at least $1 - \delta$.  The remainder of the proof is identical, replacing the bound in equation \eqref{equationmeanberryesseen} with the above bound.  
	\end{proof}

	\begin{proof}[Proof of Theorem \ref{theoremtestinglmlasso}]
		It is immediate from the definition of $\fhatlasso(\cdot)$ that $\fhatlasso(\xvec_{i} ; (\xvec_{j}, y_{j})_{j=1}^{n} ; \pi) = \fhatlasso(\xvec_{i} ; (\xvec_{j}, y_{\pi(j)})_{j=1}^{n} ; \pizero)$ for any $\pi \in \Pi$.  Now, if $\pi = \pizero$, the compatibility condition for the design is satisfied for some constant $\cc$ with probability at least $1 - \delta / 2$.  Then, it follows from Theorem 6.1 of \cite{buhlmann2011} that
		\begin{align*}
			 \p\Big\{\sum_{i=1}^{n} \langle \xvec_{i}, \betahatlasso - \betastar \rangleeuclid^{2} \leq \ctwo \lambda^{2} sn / \cc^{2} \Big\} \geq 1 - \delta/2
		\end{align*}
		for some constant $\ctwo > 0$.  Now, let $\pi \in \Pitilde$ be arbitrary and define the event 
		\begin{align*}
			\Tset \defined \Big\{ \max_{j \in [p]} \Big|\sum_{i=1}^{n} x_{i,j} y_{\pi(i)} \Big| \leq 3 \cthree^{-1/2} n^{1/2} \separamxi \log^{1/2}(6p / \delta) \Big\},
		\end{align*}
		where $x_{i,j}$ denotes the $j$th entry of $\xvec_{i}$.  Fix $j \in [p]$ and let $\xi_{i,j} \defined x_{i, j} y_{\pi(i)}$.  By the triangle inequality, we have that 
		\begin{align*}
			\Big|\sum_{i=1}^{n} \xi_{i,j} \Big|
			\leq \Big|\sum_{i \in \Asetpione} \xi_{i,j} \Big| 
			+ \Big|\sum_{i \in \Asetpitwo} \xi_{i,j} \Big|
			+ \Big|\sum_{i \in \Asetpithree} \xi_{i,j} \Big|.
		\end{align*}
		Then, by the construction of $\Asetpione$, we have that $(\xi_{i,j})_{i \in \Asetpione}$ are independent and identically distributed sub-exponential random variables with parameter $\separamxi \leq \sgparamx (\sgparammu + \sgparamepsilon)$.  By Bernstein's inequality, for any $\tone > 0$, it follows that
		\begin{align*}
			\p\Big( \Big| \sum_{i\in \Asetpione} \xi_{i,j} \Big| > \tone \Big) \leq 2 \exp\Big[ -\cthree \min(|\Asetpione|^{-1} \tone^{2}\separamxi^{-2}, \tone \separamxi^{-1}) \Big]
		\end{align*}
		for some universal constant $\cthree > 0$.  Let 
		\begin{align*}
			\tone \defined \cone^{-1/2} n^{1/2} \separamxi \log^{1/2}(6p / \delta).
		\end{align*}
		Noting that $|\Asetpione| \leq n$, we have for $n$ sufficiently large, 
		\begin{align*}
			\p\Big( \Big| \sum_{i\in \Asetpione} \xi_{i,j} \Big| > \cthree^{-1/2} n^{1/2} \separamxi \log^{1/2}(6p / \delta) \Big) \leq \delta / (3p).
		\end{align*}
		Taking a union bound shows that 
		\begin{align*}
			\p\Big( \max_{j \in [p]} \Big| \sum_{i\in \Asetpione} \xi_{i,j} \Big| > \cthree^{-1/2} n^{1/2} \separamxi \log^{1/2}(6p / \delta) \Big) \leq \delta / 3.
		\end{align*}
		A similar calculation for $\Asetpitwo$ yields
		\begin{align*}
			\p\Big( \max_{j \in [p]} \Big| \sum_{i\in \Asetpitwo} \xi_{i} \Big| > \cthree^{-1/2} n^{1/2} \separamxi \log^{1/2}(6p / \delta) \Big) \leq \delta / 3.
		\end{align*}
		Now, 
		\begin{align*}
			\Big|\sum_{i \in \Asetpithree} \xi_{i,j} \Big| \leq |\Asetpithree| \max_{j \in [p]} \max_{i \in \Asetpithree} | \xi_{i,j} |.
		\end{align*}
		Again, by Bernstein's inequality, for $n$ sufficiently large, 
		\begin{align*}
			\p\Big\{\max_{j \in [p]} \max_{i \in \Asetpithree}|\xi_{i,j}| > \cthree^{-1} \separamxi \log(6p|\Asetpithree|/\delta) \Big\} \leq \delta / 3.
		\end{align*}
		Combining the above calculations, we have that 
		\begin{align*}
			\p(\Tset) \geq 1 - \delta
		\end{align*}
		for $n$ sufficiently large.  On $\Tset$, for any $\beta \in \R^{p}$, it follows that
		\begin{align*}
			\frac{1}{n} \sum_{i=1}^{n} (y_{\pi(i)} - \langle \xvec_{i}, \beta \rangleeuclid)^{2} + \lambda \Vert \beta \Vert_{1} 
			&= \frac{1}{n} \{ \Vert \Yvecpi \Vert_{2}^{2} - 2\langle \Yvecpi, \Xmat \beta \rangleeuclid + \Vert \Xmat\beta \Vert_{2}^{2} \} + \lambda \Vert \beta \Vert_{1} \\ 
			&\geq \frac{1}{n} \{ \Vert \Yvecpi \Vert_{2}^{2} - 6 \cthree^{-1/2} n^{1/2} \separamxi \log^{1/2}(6p / \delta) \Vert \beta \Vert_{1} + \Vert \Xmat\beta \Vert_{2}^{2} \} + \lambda \Vert \beta \Vert_{1} \\ 
			&\geq \frac{1}{n} \{ \Vert \Yvecpi \Vert_{2}^{2} + \Vert \Xmat\beta \Vert_{2}^{2} \} + (\lambda - 2\lambdazero) \Vert \beta \Vert_{1}.
		\end{align*}
		Thus, the above is minimized when $\beta = \zerovec_{p}$.  Therefore, 
		\begin{align*}
			\p(\betahatlassopi = \zerovec_{p}) \geq 1 - \delta.
		\end{align*}
		Invoking Corollary \ref{corollarypermutationmeanremainder} finishes the proof.
	\end{proof}

	To facilitate the proof of Theorem \ref{theoremtestinglmlzero}, we define three auxiliary estimators.  Let 
	\begin{align*}
		\betahatlzeropiasetj \defined \argmin_{\beta \in \R^{p}, \Vert \beta \Vert = s} \sum_{i=\Asetpij} (y_{\pi(i)} - \langle \xvec_{i}, \beta \rangleeuclid)^{2}
	\end{align*}
	for $k = 1, 2, 3$.  Thus, $\betahatlzeropiasetj$ is the $L_{0}$ estimator of $\beta$ using only the data $(\xvec_{i}, y_{\pi(i)})_{i \in \Asetpij}$ for $k = 1, 2, 3$.  The following lemma relates the squared predicted values of $\betahatlzeropi$ with $(\betahatlzeropiasetj)_{k=1}^{3}$.  The result allows us to decouple the dependence between the covariates and the response by analyzing the observations in $\Asetpione$ and $\Asetpitwo$ separately.

	\begin{lemma}\label{lemmaprojectioninequality}
		Consider the model given in equation \eqref{modellm}.  Then, 
		\begin{align*}
			\sum_{i=1}^{n} \langle \xvec_{i}, \betahatlzero \rangleeuclid^{2} 
			\leq \sum_{i \in \Asetpione} \langle \xvec_{i}, \betahatlzeropiasetone \rangleeuclid^{2} 
			+ \sum_{i \in \Asetpitwo} \langle \xvec_{i}, \betahatlzeropiasettwo \rangleeuclid^{2}
			+ \sum_{i \in \Asetpithree} \langle \xvec_{i}, \betahatlzeropiasetthree \rangleeuclid^{2}.
		\end{align*}
	\end{lemma}
	
	\begin{proof}[Proof of Lemma \ref{lemmaprojectioninequality}]
		Indeed, for any $\beta \in \R^{p}$, we have that
		\begin{align*}
			\sum_{i=1}^{n} (y_{i} - \langle \xvec_{i}, \beta \rangleeuclid)^{2} 
			= \sum_{k=1}^{3} \sum_{i \in \Asetpij} (y_{i} - \langle \xvec_{i}, \beta \rangleeuclid)^{2}.
		\end{align*}
		Minimizing both sides with respect to $\beta$, it follows that 
		\begin{align*}
			\sum_{i=1}^{n} (y_{i} - \langle \xvec_{i}, \betahatlzeropi \rangleeuclid)^{2} 
% 			&= \min_{\beta \in \R^{p} : \Vert \beta \Vert_{0} = s} \sum_{i=1}^{n} (y_{i} - \langle \xvec_{i}, \beta \rangleeuclid)^{2} \\
			&= \min_{\beta \in \R^{p} : \Vert \beta \Vert_{0} = s} \sum_{k = 1}^{3} \sum_{i\in\Asetpij} (y_{i} - \langle \xvec_{i}, \beta \rangleeuclid)^{2} \\
			&\geq \sum_{k = 1}^{3} \min_{\beta \in \R^{p} : \Vert \beta \Vert_{0} = s} \sum_{i\in\Asetpij} (y_{i} - \langle \xvec_{i}, \beta \rangleeuclid)^{2} \\
			&= \sum_{k = 1}^{3} \sum_{i\in\Asetpij} (y_{i} - \langle \xvec_{i}, \betahatlzeropiasetj \rangleeuclid)^{2}.
		\end{align*}
		Applying the Pythagorean Theorem finishes the proof.
	\end{proof}

	\begin{proof}[Proof of Theorem \ref{theoremtestinglmlzero}]
		It is clear that $\fhatlzero(\xvec_{i} ; (\xvec_{j}, y_{j})_{j=1}^{n} ; \pi) = \fhatlzero(\xvec_{i} ; (\xvec_{j}, y_{\pi(j)})_{j=1}^{n} ; \pizero)$ for any $\pi \in \Pi$.  Moreover, from Theorem 2.6 of \cite{rigollet2017}, there exists a constant $\cone > 0$ such that
		\begin{align*}
			\p\Big\{\sum_{i = 1}^{n} \langle \xvec_{i}, \betahatlzero - \betastar \rangleeuclid^{2} \leq \cone\sgparamepsilon(\log{\binom{p}{2\shat}} + \log(1 / \delta)) \Big\} \geq 1 - \delta.
		\end{align*}
		Now, fix $\pi \in \Pitilde$.  Since $(\xvec_{i})_{i \in \Asetpione}$ and $(y_{\pi(i)})_{i\in\Asetpione}$ are mutually independent, Theorem 2.6 of \cite{rigollet2017} implies there exists a constant $\ctwo > 0$ such that
		\begin{align*}
			\p\Big\{\sum_{i \in \Asetpione} \langle \xvec_{i}, \betahatlzeropiasetone \rangleeuclid^{2} \leq \ctwo(\sgparammu + \sgparamepsilon)(\log{\binom{p}{2\shat}} + \log(3 / \delta)) \Big\} \geq 1 - \delta/3.
		\end{align*}
		Analogously, we see that 
		\begin{align*}
			\p\Big\{\sum_{i \in \Asetpitwo} \langle \xvec_{i}, \betahatlzeropiasettwo \rangleeuclid^{2} \leq \ctwo(\sgparammu + \sgparamepsilon)(\log{\binom{p}{2\shat}} + \log(3 / \delta)) \Big\} \geq 1 - \delta/3.
		\end{align*}
		Next, by an argument identical to that of Lemma \ref{lemmapermutationmeanremainder}, it follows that
		\begin{align*}
			\sum_{i \in \Asetpione} \langle \xvec_{i}, \betahatlzeropiasetthree \rangleeuclid^{2} \leq \sum_{i \in \Asetpione} y_{\pi(i)}^{2} \leq \log^{2}(n) (\sigmamusq + \sigmaepsilonsq) + \cthree \log(n)
		\end{align*}
		with probability at least $1 - \delta / 3$ for some constant $\cthree > 0$.  Hence, Lemma \ref{lemmaprojectioninequality} implies that
		\begin{align*}
			\sum_{i=1}^{n} (\langle \xvec_{i}, \betahatlzero \rangleeuclid)^{2} \leq 2\ctwo(\sgparammu + \sgparamepsilon)(\log{\binom{p}{2\shat}} + \log(3 / \delta)) + \log^{2}(n) (\sigmamusq + \sigmaepsilonsq) + \cthree \log(n)
		\end{align*}
		with probability at least $1 - \delta$.  The result now follows from Corollary \ref{corollarypermutationmeanremainder}.
	\end{proof}

	\begin{proof}[Proof of Theorem \ref{theoremtestingtrmlzero}]
		The proof is identical to that of Theorem \ref{theoremtestinglmlzero}, replacing Theorem 2.6 of \cite{rigollet2017} with Theorem \ref{theorempredictionrisk} of the present paper.
	\end{proof}
	
	\begin{proof}[Proof of Theorem \ref{theoremtestingtrmlzeromisspecified}]
		Let $\Thetamattilde = \Utilde \Vtilde^\T$ with $\Utilde \in \R^{\done \times \rtilde}$ and $\Vtilde \in \R^{\dtwo \times \rtilde}$, $\mutilde_{i} \defined \langle \Xmat_{i}, \Thetamattilde \ranglehs$ with variance $\sigmamutilde \defined \var(\mutilde_{1}) = \vec(\Thetamattilde)^\T \mathbf{\Sigma} \vec(\Thetamattilde)$, and $\eta_{i} \defined \langle \Xmat_{i}, \Thetastar - \Thetamattilde \ranglehs$, yielding the decomposition
		\begin{align*}
		    y_{i} = \langle \Xmat_{i}, \Thetamattilde \ranglehs + \eta_{i} + \epsilon_{i}
		    = \langle \Xmat_{i} \Vtilde, \Utilde \ranglehs + \eta_{i} + \epsilon_{i}.
		\end{align*}
		Since $\Thetamattilde$ satisfies
		\begin{align*}
		    \Thetamattilde = \argmin_{\Thetamat \in \R^{\done \times \dtwo}, \rank(\Thetamat) \leq \rtilde} \e \langle \Xmat, \Thetastar - \Thetamat \ranglehs^{2},
		\end{align*}
		it follows from the population first-order condition that
		\begin{align*}
		    \e \vec(\Xmat \Vtilde) \langle \Xmat, \Thetastar - \Thetamattilde \ranglehs = \e \vec(\Xmat \Vtilde) \eta = \mathbf{0}_{\rtilde \done},
		\end{align*}
		implying that $\vec(\Xmat_{i} \Vtilde)$ is uncorrelated with $\eta_{i}$.  
		Now, consider an auxiliary oracle estimator $\hat\bTheta_{\tilde \bV}$ given by
		\begin{align*}
			\UhatVtilde \defined \argmin_{\Umat \in \R^{\done \times \rtilde}} \sum_{i=1}^{n} (y_{i} - \langle \Xmat_{i}, \Umat \Vtilde^\T \ranglehs)^{2}~~\text{and}~~\ThetahatVtilde \defined \UhatVtilde \Vtilde^\T.
		\end{align*}
		Since $\Thetahatlzero$ is the empirical risk minimizer, it follows from the Pythagorean Theorem that 
		\begin{align*}
		    \Lambdapizero 
		    = \sum_{i=1}^{n} \langle \Xmat_{i}, \Thetahatlzero \ranglehs^{2} 
		    \geq \sum_{i=1}^{n} \langle \Xmat_{i}, \ThetahatVtilde \ranglehs^{2}
		    = \Vert \projVtilde \Yvec \Vert_{2}^{2}
		    = \Vert \mutildevec \Vert_{2}^{2} + 2\langle \mutildevec, \etavec + \epsilonvec \rangleeuclid + \Vert \projVtilde (\etavec + \epsilonvec) \Vert_{2}^{2}. 
		\end{align*}
		Fix $\delta > 0$ arbitrarily.  Now, proceeding as in the proof of Theorem \ref{theoremtestingalternative} and Corollary \ref{corollarysgmean}, there exists $\tone > 0$, depending on $\delta$, $\fnorm{\Thetamattilde}$ and $K_x$, such that with probability at least $1 - \delta / 2$, 
		\begin{align*}
		    \Vert \mutildevec \Vert_{2}^{2} \geq n\sigmamutilde - \tone n^{1/2}.
		\end{align*}
		If in addition Assumption \ref{assumptiontestingtrmkoltchinskii} is satisfied, then 
		\begin{align*}
		    \Vert \mutildevec \Vert_{2}^{2} \geq n\sigmamutilde - \tone' n^{1/2} \sigmamutilde
		\end{align*}
		with probability at least $1 - \delta / 2$, where $t_1'$ depends on $\delta$ and $\kol$.  For the second term, since $\vec(\Xmat_{i} \Vtilde)$ is uncorrelated with $\eta_{i}$, it follows that $\e\langle \mutildevec, \etavec + \epsilonvec \rangleeuclid = 0$.
		Now, by Chebyshev's inequality, there exists $\ttwo > 0$ depending only on $\delta$ such that
		\begin{align*}
		    2 |\langle \mutildevec, \etavec + \epsilonvec \rangleeuclid| \leq \ttwo n^{1/2} \sigma_{\mutilde} (\sigmamu + \sigmaepsilon)
		\end{align*}
		with probability at least $1 - \delta / 2$. Therefore, with probability at least $1 - \delta$,
		\[
		    \Lambdapizero \ge n\sigmamutilde - t_1n^ {1 / 2} - \ttwo n^{1/2} \sigma_{\mutilde} (\sigmamu + \sigmaepsilon). 
		\]
		
		It remains to bound $\Lambdapi$ for $\pi \in \Pitilde$.  Define the auxiliary estimators
		\begin{align*}
		    \Thetahatlzero^{(\Asetpij)} \defined \argmin_{\Thetamat \in \R^{\done \times \dtwo}, \rank(\Thetamat) \leq \rtilde}  \sum_{i \in \Asetpij} (y_{\pi(i)} - \langle \Xmat_{i}, \Thetamat \ranglehs)^{2}.
		\end{align*}
		By an identical argument as in Lemma \ref{lemmaprojectioninequality}, it follows that
		\begin{align*}
		    \Lambdapi = \sum_{i=1}^{n} \langle \Xmat_{i}, \Thetahatlzeropi \ranglehs^{2}
		    \leq \sum_{k=1}^{3} \sum_{i \in \Asetpij} \langle \Xmat_{i}, \Thetahatlzero^{(\Asetpij)} \ranglehs^{2}.
		\end{align*}
		By Theorem \ref{theorempredictionrisk}, there exists a constant $\tthree$, depending on $\delta$, $\sigmamusq$, $\sigmaepsilon$, $\sgparamx$, and $\sgparamepsilon$, such that 
		\begin{align*}
		    \sum_{i \in \Asetpij} \langle \Xmat_{i}, \Thetahatlzero^{(\Asetpij)} \ranglehs^{2} \leq \tthree \rtilde d \log(n)
		\end{align*}
		for $j = 1, 2$ with probability at least $1 - \delta / 3$.  Similarly, following Lemma \ref{lemmapermutationmeanremainder}, we have that 
		\begin{align*}
		    \sum_{i \in \Asetpij} \langle \Xmat_{i}, \Thetahatlzero^{(\Asetpij)} \ranglehs^{2} \leq \log^{2}(n) (\sigmamusq + \sigmaepsilonsq) + \tfour \log(n)
		\end{align*}
		with probability at least $1 - \delta / 3$ for a constant $\tfour$ depending on $\delta$, $\fnorm{\Thetastar}$, $\sgparamx$, and $\sgparamepsilon$.  Combining, we have 
		\begin{align*}
		    \Lambdapi \leq 2\tthree \rtilde d \log(n) + \log^{2}(n) (\sigmamusq + \sigmaepsilonsq) + \tfour \log(n)
		\end{align*}
		with probability at least $1 - \delta$.  Proceeding as in Theorem \ref{theoremtestingalternative} finishes the proof.
	\end{proof}

%% file: acknowledgements.tex
We thank Professor Zili Zhang at Tongji University for his valuable suggestions and comments as we were developing Theorem \ref{theoremcoveringnumber}.  We also thank Professors Xuming He, Long Nguyen, and Stilian Stoev at the University of Michigan, Ann Arbor for their constructive comments of our work.  

%% file: funding.tex
ML is supported in part by NSF Grant DMS-1646108.  YR is supported in part by NSF Grants DMS-1712962 and DMS-2113364.  RZ is supported in part by NSF Grants DMS-1856541 and DMS-1926686 and the Ky Fan and Yu-Fen Fan Endowment Fund at the Institute for Advanced Study.  ZZ is supported in part by NSF Grant DMS-2015366.

%% file: supplement.tex
In the supplementary material, we provide additional simulation results from Section \ref{sectionsimulations} for the matrix completion setting.

\setcounter{table}{0}
\renewcommand{\thetable}{S\arabic{table}}
\setcounter{figure}{0}
\renewcommand{\thefigure}{S\arabic{figure}}

\input{RTEstMC}
\begin{figure}[h]
    \centering
    \includegraphics[scale=0.145]{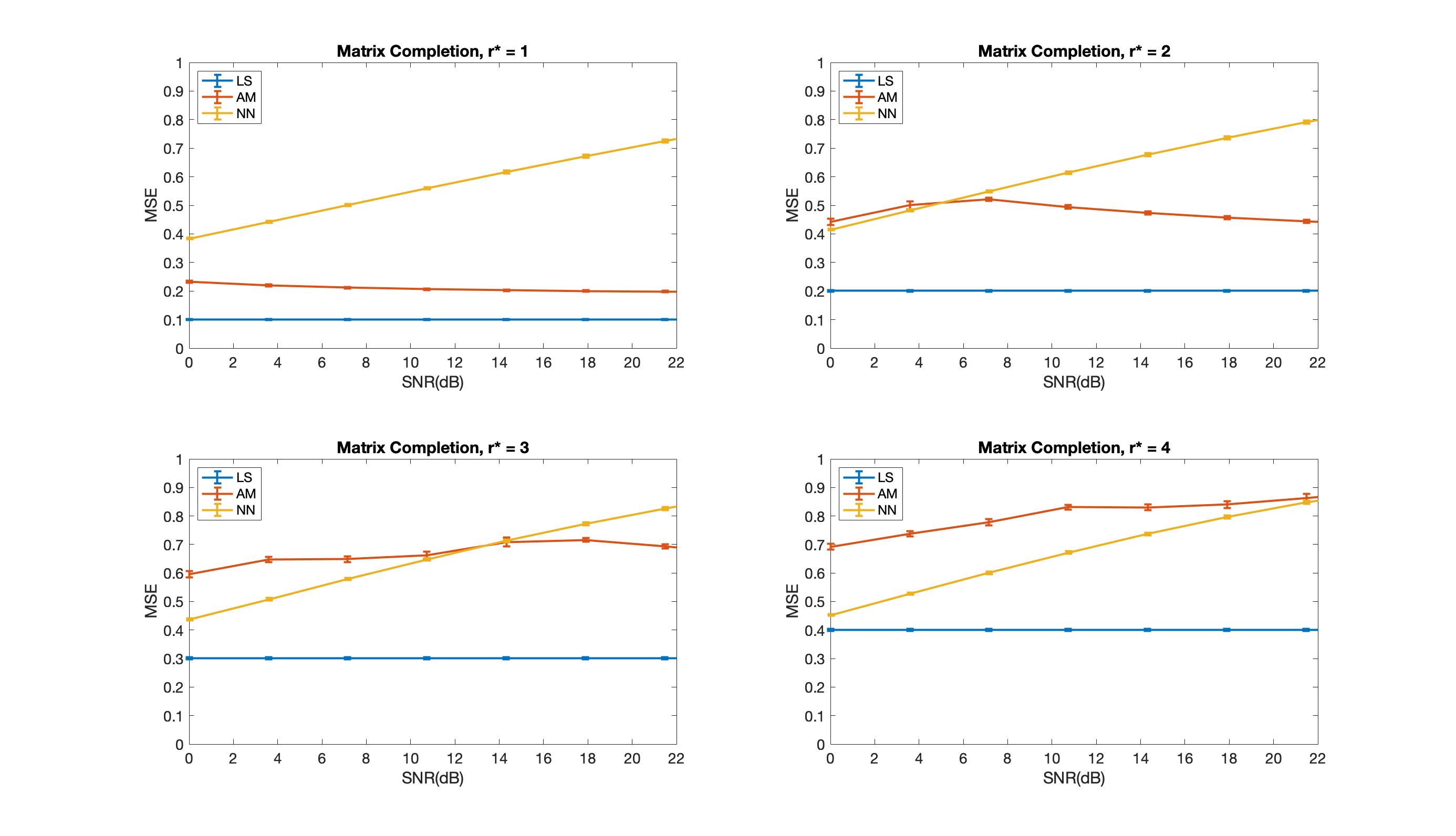}
    \caption{Plots of in-sample prediction error for matrix completion}
    \label{plot_est_mc}
\end{figure}

\input{RTInfTRMmc_CV}
\begin{figure}[h]
    \centering
    \includegraphics[scale=0.2]{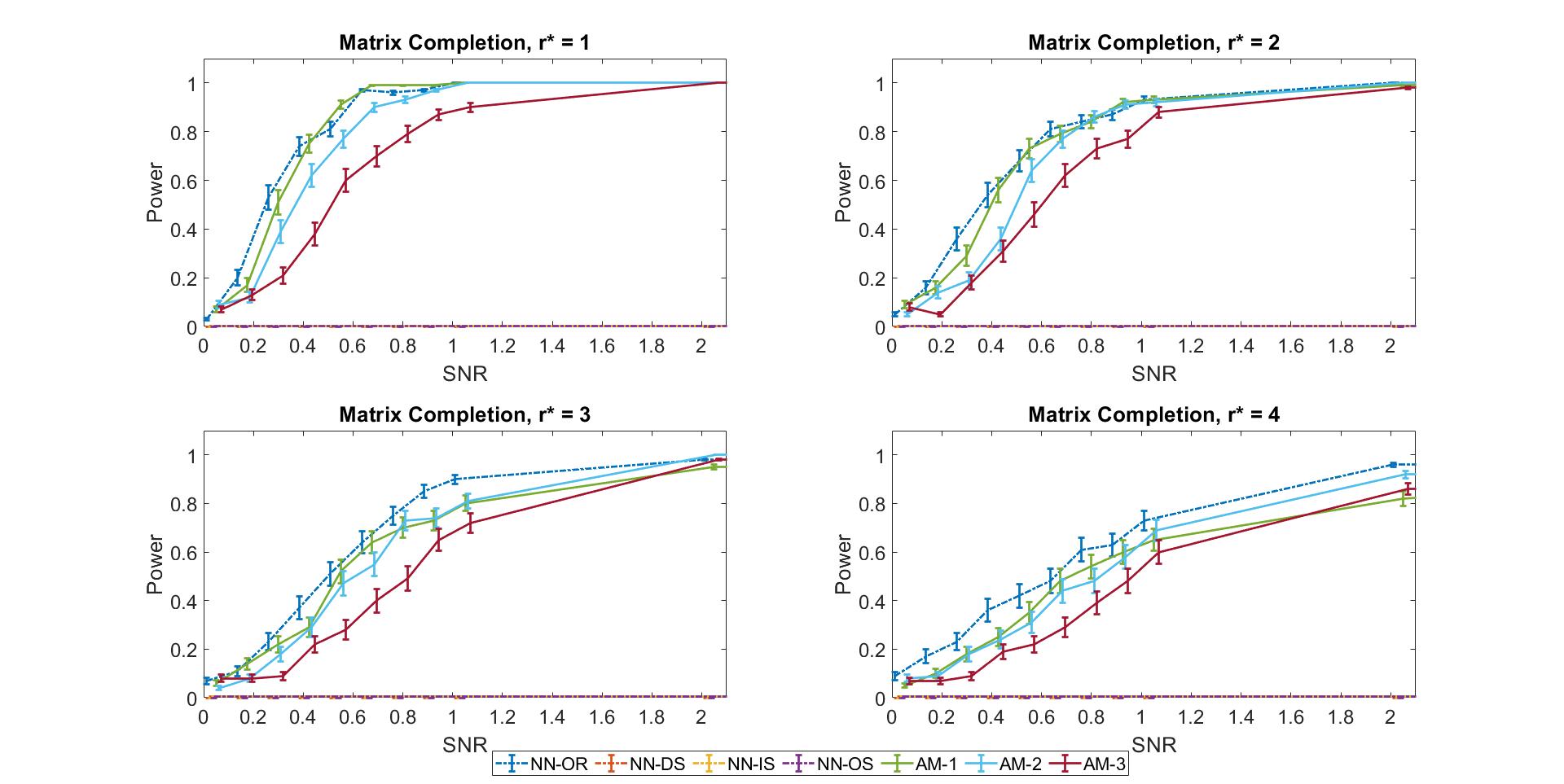}
    \caption{Plots of power for matrix completion}
    \label{plot_inf_mc}
\end{figure}

%% file: RTEstMC.tex
% latex table generated in R 4.1.0 by xtable 1.8-4 package
% Thu Nov 25 09:35:28 2021
\begin{table}[h]
\centering
\caption{Simulations for In-Sample Prediction Risk for Matrix Completion} 
\label{tableestmc}
\begin{tabular}{|l|l|rrrrrrrrrr|}
   \hline
 & {SNR} & 1.00 & 1.43 & 2.04 & 2.92 & 4.18 & 5.98 & 8.55 & 12.23 & 17.48 & 25.00 \\ 
   \hline
 & LS & 0.10 & 0.10 & 0.10 & 0.10 & 0.10 & 0.10 & 0.10 & 0.10 & 0.10 & 0.10 \\ 
  $r^\ast = 1$ & AM & 0.23 & 0.22 & 0.21 & 0.21 & 0.20 & 0.20 & 0.20 & 0.20 & 0.20 & 0.20 \\ 
   & NN & 0.38 & 0.44 & 0.50 & 0.56 & 0.62 & 0.67 & 0.73 & 0.77 & 0.82 & 0.86 \\ 
   \hline
 & LS & 0.20 & 0.20 & 0.20 & 0.20 & 0.20 & 0.20 & 0.20 & 0.20 & 0.20 & 0.20 \\ 
  $r^\ast = 2$ & AM & 0.44 & 0.50 & 0.52 & 0.49 & 0.47 & 0.46 & 0.44 & 0.43 & 0.43 & 0.42 \\ 
   & NN & 0.42 & 0.48 & 0.55 & 0.61 & 0.68 & 0.74 & 0.79 & 0.84 & 0.88 & 0.92 \\ 
   \hline
 & LS & 0.30 & 0.30 & 0.30 & 0.30 & 0.30 & 0.30 & 0.30 & 0.30 & 0.30 & 0.30 \\ 
  $r^\ast = 3$ & AM & 0.60 & 0.65 & 0.65 & 0.66 & 0.71 & 0.72 & 0.69 & 0.67 & 0.65 & 0.64 \\ 
   & NN & 0.44 & 0.51 & 0.58 & 0.65 & 0.71 & 0.77 & 0.83 & 0.87 & 0.91 & 0.94 \\ 
   \hline
 & LS & 0.40 & 0.40 & 0.40 & 0.40 & 0.40 & 0.40 & 0.40 & 0.40 & 0.40 & 0.40 \\ 
  $r^\ast = 4$ & AM & 0.69 & 0.74 & 0.78 & 0.83 & 0.83 & 0.84 & 0.86 & 0.89 & 0.87 & 0.86 \\ 
   & NN & 0.45 & 0.53 & 0.60 & 0.67 & 0.74 & 0.80 & 0.85 & 0.89 & 0.93 & 0.95 \\ 
   \hline
\end{tabular}
\end{table}

%% file: RTInfTRMmc_CV.tex
% latex table generated in R 4.1.0 by xtable 1.8-4 package
% Thu Mar 24 11:25:57 2022
\begin{table}[h]
\centering
\caption{Simulations for Inference for Matrix Completion} 
\label{tableinftrmmc}
\begin{tabular}{|l|l|rrrrrrrrrr|}
   \hline
 & SNR & 0.000 & 0.125 & 0.250 & 0.375 & 0.500 & 0.625 & 0.750 & 0.875 & 1.000 & 2.000  \\ 
   \hline
 & FT & 0.08 & 0.86 & 1.00 & 1.00 & 1.00 & 1.00 & 1.00 & 1.00 & 1.00 & 1.00  \\ 
   & LS & 0.06 & 0.80 & 0.98 & 1.00 & 1.00 & 1.00 & 1.00 & 1.00 & 1.00 & 1.00  \\ 
   & NN-OR & 0.03 & 0.20 & 0.53 & 0.74 & 0.81 & 0.97 & 0.96 & 0.97 & 1.00 & 1.00  \\ 
   & NN-DS & 0.00 & 0.00 & 0.00 & 0.00 & 0.00 & 0.00 & 0.00 & 0.00 & 0.00 & 0.00  \\ 
   & NN-IS & 0.00 & 0.00 & 0.00 & 0.00 & 0.00 & 0.00 & 0.00 & 0.00 & 0.00 & 0.00  \\ 
  $r^\ast = 1$ & NN-OS & 0.00 & 0.00 & 0.00 & 0.00 & 0.00 & 0.00 & 0.00 & 0.00 & 0.00 & 0.00 \\ 
   & AM-1 & 0.07 & 0.17 & 0.51 & 0.75 & 0.91 & 0.99 & 0.99 & 0.99 & 1.00 & 1.00  \\ 
   & AM-2 & 0.09 & 0.12 & 0.39 & 0.62 & 0.77 & 0.90 & 0.93 & 0.97 & 1.00 & 1.00 \\ 
   & AM-3 & 0.07 & 0.13 & 0.21 & 0.38 & 0.60 & 0.70 & 0.79 & 0.87 & 0.90 & 1.00 \\ 
   & AM-4 & 0.06 & 0.05 & 0.15 & 0.21 & 0.30 & 0.41 & 0.55 & 0.69 & 0.66 & 0.95  \\ 
%   & AM-5 & 0.09 & 0.04 & 0.07 & 0.13 & 0.15 & 0.17 & 0.32 & 0.24 & 0.36 & 0.69 & 0.98 \\ 
   \hline
 & FT & 0.08 & 0.65 & 0.98 & 1.00 & 1.00 & 1.00 & 1.00 & 1.00 & 1.00 & 1.00  \\ 
   & LS & 0.08 & 0.61 & 0.92 & 0.99 & 1.00 & 1.00 & 1.00 & 1.00 & 1.00 & 1.00 \\ 
   & NN-OR & 0.05 & 0.16 & 0.36 & 0.54 & 0.68 & 0.81 & 0.84 & 0.87 & 0.93 & 1.00 \\ 
   & NN-DS & 0.00 & 0.00 & 0.00 & 0.00 & 0.00 & 0.00 & 0.00 & 0.00 & 0.00 & 0.00  \\ 
   & NN-IS & 0.00 & 0.00 & 0.00 & 0.00 & 0.00 & 0.00 & 0.00 & 0.00 & 0.00 & 0.00 \\ 
  $r^\ast = 2$ & NN-OS & 0.00 & 0.00 & 0.00 & 0.00 & 0.00 & 0.00 & 0.00 & 0.00 & 0.00 & 0.00 \\ 
   & AM-1 & 0.09 & 0.16 & 0.29 & 0.56 & 0.73 & 0.79 & 0.84 & 0.92 & 0.93 & 0.99 \\ 
   & AM-2 & 0.05 & 0.14 & 0.19 & 0.36 & 0.64 & 0.77 & 0.86 & 0.91 & 0.92 & 1.00  \\ 
   & AM-3 & 0.08 & 0.05 & 0.18 & 0.31 & 0.46 & 0.62 & 0.73 & 0.77 & 0.88 & 0.98  \\ 
   & AM-4 & 0.08 & 0.09 & 0.15 & 0.11 & 0.20 & 0.25 & 0.43 & 0.48 & 0.52 & 0.94  \\ 
%   & AM-5 & 0.10 & 0.06 & 0.07 & 0.07 & 0.12 & 0.16 & 0.20 & 0.27 & 0.31 & 0.55 & 0.96 \\ 
   \hline
 & FT & 0.08 & 0.52 & 0.93 & 1.00 & 1.00 & 1.00 & 1.00 & 1.00 & 1.00 & 1.00  \\ 
   & LS & 0.06 & 0.44 & 0.87 & 0.97 & 1.00 & 1.00 & 1.00 & 1.00 & 1.00 & 1.00  \\ 
   & NN-OR & 0.07 & 0.11 & 0.23 & 0.37 & 0.51 & 0.64 & 0.75 & 0.85 & 0.90 & 0.98  \\ 
   & NN-DS & 0.00 & 0.00 & 0.00 & 0.00 & 0.00 & 0.00 & 0.00 & 0.00 & 0.00 & 0.00  \\ 
   & NN-IS & 0.00 & 0.00 & 0.00 & 0.00 & 0.00 & 0.00 & 0.00 & 0.00 & 0.00 & 0.00  \\ 
  $r^\ast = 3$ & NN-OS & 0.00 & 0.00 & 0.00 & 0.00 & 0.00 & 0.00 & 0.00 & 0.00 & 0.00 & 0.00 \\ 
   & AM-1 & 0.06 & 0.14 & 0.22 & 0.29 & 0.52 & 0.64 & 0.70 & 0.73 & 0.80 & 0.95  \\ 
   & AM-2 & 0.04 & 0.08 & 0.18 & 0.29 & 0.47 & 0.55 & 0.73 & 0.74 & 0.81 & 1.00  \\ 
   & AM-3 & 0.08 & 0.08 & 0.09 & 0.22 & 0.28 & 0.40 & 0.49 & 0.65 & 0.72 & 0.98  \\ 
   & AM-4 & 0.09 & 0.06 & 0.09 & 0.12 & 0.14 & 0.24 & 0.25 & 0.36 & 0.42 & 0.86  \\ 
%   & AM-5 & 0.12 & 0.06 & 0.07 & 0.05 & 0.06 & 0.12 & 0.16 & 0.22 & 0.19 & 0.48 & 0.93 \\ 
   \hline
 & FT & 0.08 & 0.36 & 0.79 & 0.98 & 1.00 & 1.00 & 1.00 & 1.00 & 1.00 & 1.00  \\ 
   & LS & 0.06 & 0.31 & 0.67 & 0.97 & 0.99 & 1.00 & 1.00 & 1.00 & 1.00 & 1.00  \\ 
   & NN-OR & 0.09 & 0.17 & 0.23 & 0.36 & 0.42 & 0.48 & 0.61 & 0.63 & 0.73 & 0.96  \\ 
   & NN-DS & 0.00 & 0.00 & 0.00 & 0.00 & 0.00 & 0.00 & 0.00 & 0.00 & 0.00 & 0.00  \\ 
   & NN-IS & 0.00 & 0.00 & 0.00 & 0.00 & 0.00 & 0.00 & 0.00 & 0.00 & 0.00 & 0.00  \\ 
  $r^\ast = 4$ & NN-OS & 0.00 & 0.00 & 0.00 & 0.00 & 0.00 & 0.00 & 0.00 & 0.00 & 0.00 & 0.00  \\ 
   & AM-1 & 0.05 & 0.10 & 0.18 & 0.25 & 0.35 & 0.48 & 0.54 & 0.60 & 0.65 & 0.82  \\ 
   & AM-2 & 0.08 & 0.09 & 0.18 & 0.24 & 0.31 & 0.44 & 0.48 & 0.58 & 0.69 & 0.92  \\ 
   & AM-3 & 0.07 & 0.07 & 0.09 & 0.19 & 0.22 & 0.29 & 0.39 & 0.48 & 0.60 & 0.86  \\ 
   & AM-4 & 0.05 & 0.04 & 0.09 & 0.07 & 0.13 & 0.15 & 0.30 & 0.31 & 0.40 & 0.75  \\ 
%   & AM-5 & 0.12 & 0.07 & 0.08 & 0.08 & 0.07 & 0.08 & 0.12 & 0.19 & 0.19 & 0.46 & 0.86 \\ 
   \hline
\end{tabular}
\end{table}